\documentclass[a4paper,10pt]{article}
\usepackage{geometry,times}
\usepackage{amsfonts}
\usepackage{amssymb}
\usepackage{amsmath}
\usepackage{tikz}
\usepackage{pgfplots}
\usepackage{tikz-3dplot}
\usetikzlibrary{3d}
\usetikzlibrary{shapes.geometric}
\tdplotsetmaincoords{60}{115}
\pgfplotsset{compat=newest}
\definecolor{paleyellow}{rgb}{1, 1, 0.9}
\usepackage{amsthm}

\newcommand{\bb}{\mathbb}

\theoremstyle{plain}
\newtheorem{satz}{prop}
\newtheorem{corollary}[satz]{Corollary}

\newtheorem{definition}[satz]{Definition}
\newtheorem{prop}[satz]{Proposition}
\theoremstyle{thm}
\newtheorem{theorem}[satz]{Theorem}
\theoremstyle{note}
\newtheorem{note}[satz]{Note}
\newtheorem{remark}[satz]{Remark}
\theoremstyle{problem}

\theoremstyle{notadef}
\newtheorem{notation}[satz]{Notation}
\usepackage[]{graphicx}
\usepackage{color}
\definecolor{dark-BLGR}{rgb}{0.28,0.28,0.28}

\definecolor{dark-red}{rgb}{0.90,0.14,0.14}

\definecolor{dark-viol}{rgb}{0.73,0.18,0.69}
\definecolor{Cyan}{rgb}{0.31,0.67,0.82}

\newcommand{\red}{\textcolor{red}}

\definecolor{light-yellow}{rgb}{1,1,0.8}

\definecolor{Blue}{rgb}{0.0353,0.0275,0.4}

\definecolor{GRAY}{rgb}{0.26,0.26,0.26}

\definecolor{GREYY}{rgb}{0.45,0.45,0.45}

\definecolor{vivid-viol}{rgb}{0.3255,0.0353,0.55}

\definecolor{dark-blue}{rgb}{0.05,0.05,0.65}

\definecolor{dark-green}{rgb}{0.03,0.77,0.29}

\definecolor{dark-Green}{rgb}{0.03,0.57,0.09}

\definecolor{strong-viol}{rgb}{0.2353,0.094,0.349}

\definecolor{BLUE}{rgb}{0.41,0.44,0.93}

\definecolor{RED}{rgb}{0.90,0.18,0.32}

\definecolor{Yel}{rgb}{0.89,0.95,0.19}

\definecolor{White}{rgb}{1,1,1}
\newcommand{\White}{\textcolor{White}}
\definecolor{black}{rgb}{0,0,0}

\definecolor{Orchid}{rgb}{0.6,0.1607,0.2}

\definecolor{Orange}{rgb}{0.99,0.521,0.34902}

\definecolor{magenta}{rgb}{0.99,0.02,0.99}

\definecolor{Thistle}{rgb}{0.1151,0.1249,0.9873}

\definecolor{brown}{rgb}{0.5450,0.2825,0.07453}

 \usepackage{titlesec}
\titleformat{\section}
{\normalfont
  \bfseries
}
{\thesection}{1ex}{}
\titleformat{\subsection}
{\normalfont
  \bfseries
}
    {\thesubsection}{1ex}{}

\newcommand{\NN}{\ensuremath{\mathbb{N}}}

\begin{document}
\hspace*{1cm}\\[.4cm]
{\large \textbf{Estimates of the Poincar\'e constants for bounded domains in 
Euclidian space ${\bb R}^{n}$}}
\\[.5cm]
Bernd Rummler\\
\hspace*{2mm}{\small \em Otto-von-Guericke-Universit\"at Magdeburg,
  Inst. f\"ur Analysis und Numerik, PF
  4120, 39016 Magdeburg}
\\[2mm]
Gudrun Th\"ater\footnote{Corresponding author\quad
  E-mail:~\textsf{gudrun.thaeter@kit.edu}}\\
\hspace*{2mm}{\small \em Inst. f\"ur Angewandte \& Numerische Mathematik, KIT, 
  76128 Karlsruhe}
\\[5mm]
\hspace*{.1cm}\hfill\parbox{16cm}{
 \\[2mm]
{\small\textbf{Abstract:}
We study bounded, simply connected, convex, and  ${\cal C}^2$-smooth
domains ${\Omega}\,={\Omega}_{(n)}$ in the  
space ${\bb R}^{n}$, $n\,\in\,\{2,..,N\}$ with $N\,<\,\infty$. For such domains 
we use the diameter $\text{diam}\,\Omega$ 
to find balls
with the property that they 
are included in this ball, while the diameter
of the ball is chosen to be minimal as $2 \cdot R\,=\,
\text{diam}\,{\Omega}_{(n)}$. 
The center of this chosen ball is denoted by $\tilde{\underline{x}}$.
With a simple transformation we are able to map ${\Omega}_{(n)}$ onto a
domain $\hat{\Omega}_{(n)}$ into the open unit ball
${\Omega}^{*}_{(n)}:=\{{\underline{x}} \in {\bb R}^{n}:\, 
\|{\underline{x}}\|\,<1\}$ with the center ${\underline{x}_o}\,=\,{\underline{0}}$
and $R_o\,=\,1$.
We derive estimates for the Poincar\'e constants of scalar functions
$v$ and vector functions ${\underline{v}}$ on those  
domains. 
The main ingredient is  
the first eigenvalue of the Laplacian or the Stokes operator  
on ${\Omega}$ with vanishing Dirichlet traces on  $\partial{\Omega}$,
respectively 
(we use results of \cite{RuRuTh2026}). 
Additionally we apply the relation of the first eigenvalues
and one eigenfunction of the (scalar) Laplace and the Stokes operator cf. \cite{RuRuTh2026}. 
Our estimates for the Poincar\'e constants
are new. 
The dimension $n$ of the space 
has significant effect on the quality of our estimates. This is
illustrated with the help of an example.
}}
\newcommand{\abs}[1]{\left | #1 \right |}
\newcommand{\RR}{\mathbb{R}}
\newcommand{\cC}{{\mathcal{C}}}
\section{Introduction\label{sec_int}}
In  \cite{RuRuTh2026} we calculated the first simple
eigenvalues $\lambda_{1,L}$ of the Laplacian ${\boldsymbol L}={\boldsymbol L}_{o}$ on the open unit ball ${\Omega}^{*}_{(n)}$ 
in ${\bb R}^{n}$, $n\,\in\,\{2,..,N\}$ with $N\,<\,\infty$ 
and the 
of the first eigenvalues $\lambda_{1,S }$ of the Stokes operator ${\boldsymbol S}={\boldsymbol S}_{o}$ on ${\Omega}^{*}_{(n)}$
with vanishing Dirichlet traces on $\partial{\Omega}^{*}_{(n)}$ as the squares of the smallest positive solution (the smallest positive root) 
$\kappa_{1,L}$ resp. as the squares of the smallest positive solution (the smallest positive root) $\kappa_{1,S}$ 
of the equations
$J_{\frac{n}{2}-1}(\kappa_{L})\,=\,0$\,  resp. $J_{\frac{n}{2}}(\kappa_{S})\,=\,0$, for all $n\,>\,1$ (see the Note 12
and Corollary 20 in \cite{RuRuTh2026}). There one easily reads off the Poincar\'e constants for ${\Omega}^{*}_{(n)}$
as $c_{p}\,=\,{\kappa_{1,L}}^{-1}$
resp. $c_{p,S}\,=\,{\kappa_{1,S}}^{-1}$. By the shift-Theorem
(i.e. Theorem 22 in \cite{RuRuTh2026}) we have the connection of the
Poincar\'e constants for scalar functions (i.e. with the Laplacian
${\boldsymbol L}={\boldsymbol L}_{o}$ on ${\Omega}^{*}_{(n+2)}$) and
the Poincar\'e constants for vector fields (i.e. the Stokes operator 
${\boldsymbol S}={\boldsymbol S}_{o}$ on ${\Omega}^{*}_{(n)}$) as
$c_{p,(n+2)}\,=\,c_{p,S,(n)}$ for any $n\, >\,1$.\\ 
\noindent Our {\bf main} results {{ use 
the operating system \cite{OS}. We provide estimates 
adapted to any bounded domain $\Omega_{(n)}\,\subset {\bb R}^{n}$ ($n\,\in\,\{2,..,N\}$)
for scalar fields 
as well as for solenoidal vector fields by the use of the exact bounds for the Poincar\'e constants on balls (see \cite{RuRuTh2026}). There
we always assume $\Omega_{(n)}$ to be simply connected and convex  of the class  ${\cal{C}}^{2}$ resp. ${\cal{C}}^{0,1}$.}}\\[.1cm]
\begin{figure}[ht]
\hspace*{2cm}
\begin{minipage}[t]
{0.5\textwidth}
\begin{tikzpicture}[scale=1.6,tdplot_main_coords] 
\coordinate (P) at ({0},{2.2225/sqrt(2)},{2.2225/sqrt(2)
});
\shade[ball color = yellow,
    opacity = 0.35
] (0,0,0) circle (1.75cm);
\shade[ball color = blue,
    opacity = 0.4
] (0,0,0) ellipse (1.75cm and 1.4cm);
 
\draw[thick,black, -stealth] (0,0,0) -- (P);
\draw(-.3,1.,1.)  node[anchor=east,black]{$R=\text{diam}\,\Omega_{(n)}$};
\draw(0.25,-0.55,-0.35) node[anchor=east]{\Large{$\Omega_{(n)}$}};
%
\end{tikzpicture}
\end{minipage}
\hspace*{-1cm}
\begin{minipage}[t]{0.5\textwidth}
\begin{tikzpicture}[scale=1.6,tdplot_main_coords] 
\coordinate (P) at ({0},{1.905/sqrt(2)},{1.905/sqrt(2)
});
\shade[ball color = yellow,
    opacity = 0.2
] (0,0,0) circle (1.5cm);
\shade[ball color = red,
    opacity = 0.4
] (0,0,0) ellipse (1.5cm and 1.2cm);
 
 
 
\draw[thick,vivid-viol, -stealth] (0,0,0) -- (P);
\draw(-.10,1.05,1.05)  node[anchor=east,vivid-viol]{$R_o\,=\,1$};
\draw(0.25,-0.35,-0.35) node[anchor=east]{{\large${\hat{\Omega}_{(n)}}$}};
%
\end{tikzpicture}
\end{minipage}
\caption{$2$d Sketch of the domain ${\Omega}_{(n)}$ included in the ball  $\tilde{\underline{x}}\,+\,{\Omega}^{*}_{R,(n)}$ with radius  $R=\text{diam}\,\Omega_{(n)}$
and its image ${\hat{\Omega}}_{(n)}\,\subset {\Omega}_{(n)}^{*} $ 
}
\label{fig:geometry}
\end{figure}
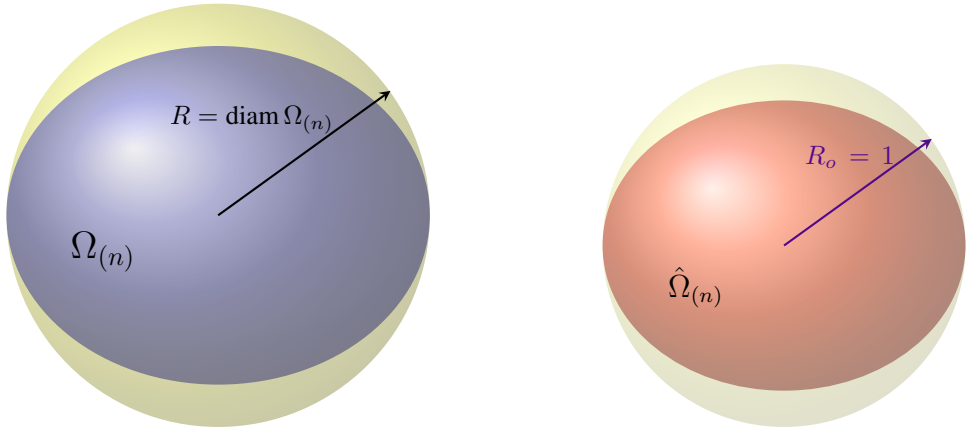
%
%
Our ulterior motive to apply 
Courant's max-min principle for the two domains $\hat{\Omega}_{(n)}$ and ${\Omega}^{*}_{(n)}$
in the sense of \cite[Satz 3 on page 355 in Bd. 1]{CouHil}). There we use simple (backward and forward) transformations to map 
${\Omega}\,={\Omega}_{(n)}$ on $\hat{\Omega}_{(n)}\,\subset\, {\Omega}^{*}_{(n)}$. Finally we calculate the estimates for 
the Poincar\'e constants on ${\Omega}_{(n)}$ by the transformation formula.\\
\noindent{\bf General notation A.} Let ${\bb R}^{n}$ be endowed with the
usual Euclidian norm $\| . \|$. 
Elements of ${\bb R}^{n}$ are denoted by underlined small letters.
We write ${\Omega}^*_{(n),R}:=\{{\underline{x}} \in {\bb R}^{n}:\,
\|{\underline{x}}\|\,
<R\}$ for the open balls and especially ${\Omega}^*_{(n)}\,:=\,\{{\underline{x}} \in {\bb R}^{n}:\,
\|{\underline{x}}\|\,
<1=R_o\}$ for the open unit balls. We use $\omega_{(n)}\,:=\,\{{\underline{x}}\in {\bb
  R}^{n}:\,\|{\underline{x}}\|=1\}$ for the closed surfaces of the open unit balls.
For $r\,\in\,(0,\infty)$ the closed spherical surfaces around the origin with radius $r$ are
$\omega_{(n),r}\,:=\,\{{\underline{x}}\in {\bb
  R}^{n}:\,\|{\underline{x}}\|=r\} $
and the surface areas of $\omega_{(n)}$ are 
$|\omega_{(n)}|\,=\,{2\pi}^{\frac{n}{2}}/
{\Gamma(\frac{n}{2})}$ ($\forall  n\,\in\,{\bb N}$), where ${\Gamma(.)}$ is the ${\Gamma}$-function.
The volume of ${\Omega}^*_{(n)}$ is $|{\Omega}^*_{(n)}|\,=\,\frac{|\omega_{(n)}|}{n}$.\\[1mm]
\noindent{\bf Domains and Transformations.}
We use a non-dimensional setting for the open bounded, convex set ${\Omega}_{(n)}$ of class ${\cal C}^{2}$ in ${\bb R}^{n}$ with 
$n\,\in\,{\bb N}$ and the diameter $\text{diam}\,{\Omega}_{(n)}\,=\,2 \cdot R$. The set ${\Omega}_{(n)}$ is included in the ball around a fixed point
$\tilde{\underline{x}}\,\in {\bb R}^{n}$:\\[.15cm]
\hspace*{1cm}
$\tilde{\Omega}^*_{\,\tilde{\underline{x}},R,(n)}\,:=\,
\tilde{\underline{x}}\,+\,{\Omega}^{*}_{R,(n)}\,:=\{{\underline{x}} \in {\bb R}^{n}:\,
\|{\underline{x}}-\tilde{\underline{x}}\|\,<\,R\,=\, \frac{1}{2}\text{diam}\,{\Omega}_{(n)}\}\,,$ \,\,
with the radius $R\,=\, \frac{1}{2}\text{diam}\,{\Omega}_{(n)}$.\\
We apply the simple transformation $F(.)$ of the closed ball ${\overline{\tilde{\Omega}^{{\!{~}}^{\scriptstyle{*}}}}
}_{\hspace*{-.1cm}\!\tilde{\underline{x}},R,(n)}$ onto the closed unit ball 
${\overline{\White{\hat\Omega}^{{\!{~}}^{\scriptstyle{*}}}\hspace*{-.45cm}
\Omega}}_{(n)}$ defined as:
\begin{align} \label{transf} 
\underline{F}:\,{\overline{\tilde{\Omega}^{{\!{~}}^{\scriptstyle{*}}}}
}_{\hspace*{-.1cm}\!\tilde{\underline{x}},R,(n)}
\,\rightarrow \,{\overline{\White{\hat\Omega}^{{\!{~}}^{\scriptstyle{*}}}\hspace*{-.45cm}
\Omega}}_{(n)}\,\quad :\,\,\,\, \underline{F}(\underline{x})\,:=\,\frac{1}{R}({\underline{x}}-\tilde{\underline{x}})\,\quad\,
\forall \,{\underline{x}}\,\in\,{\overline{\tilde{\Omega}^{{\!{~}}^{\scriptstyle{*}}}}
}_{\hspace*{-.1cm}\!\tilde{\underline{x}},R,(n)}\,.
\end{align}
\noindent{\bf General notation B.} Let ${\breve{\Omega}}$ stand as shorthand for any of the domains 
$\Omega_{(n)}$, $\tilde{\Omega}^*_{\,\tilde{\underline{x}},R,(n)}$, 
${\hat{\Omega}}_{(n)}$ and 
${\Omega}^*_{(n)}$
introduced above
and the abbreviation $(.)$ for $({\breve{\Omega}})$, respectively. 
We consider the usual Lebesgue and Sobolev spaces ${\bb L}_{2}(.)$ and
${\bb W}_{2}^{k}(.)$ 
of scalar functions and 
${\underline{\bb L}}_{{\;\!}{2}}(.)=({\bb L}_{2}(.))^{n}$
and ${\underline{\bb W}}_{{\;\!}{2}}^{k}(.)=({\bb
  W}_{2}^{k}(.))^{n}$ of vector functions. 
The norm in ${\bb L}_{2}(.)$ is denoted by $\| . \|_{2}$,
${\bb W}_{2}^{1}\hspace{-.62cm}{~}^{{~}^{{~}^{o}}}\hspace{.2cm}(.)$ is
the closure of $C_{o}^{\infty}(.)$  
in ${\bb W}_{2}^{1}(.)$.  
All solenoidal vector functions belonging to 
${\underline{C}}_{{\;\!}{o}}^{\infty}(.)$ form $\underline{\cal V}(.)$. The closures
of $\underline{\cal V}(.)$ in ${\underline{\bb L}}_{{\;\!}{2}}(.)$ 
and
${\underline{\bb W}}_{{\;\!}{2}}^{1}(.)$, respectively, are denoted by
${\underline{\bb H}}(.)$ 
and ${\underline{\bb V}}(.)$, respectively.
We use the spherical
Bessel functions ${J_{k}(.)}$ of the first kind 
of order  $k\,\in\, \{\frac{1}{2}+m,\,m\,\in\, {\bb N}_{o}\} $ for all the open balls (e.g., \cite{AAR}, \cite{CouHil}, \cite{Triebel}
etc.).
\begin{notation}\label{Nabla} 
We call the vectorial differential expression
\begin{equation} \label{nab_def}
{\underline{\nabla}}\,:=\,(
\frac{\displaystyle{\partial}}{\displaystyle{\partial x_{1}}},
\frac{\displaystyle{\partial}}{\displaystyle{\partial x_{2}}},\dots,
\frac{\displaystyle{\partial}}{\displaystyle{\partial x_{n}}})^{T}
\end{equation}
the Nabla operator (here ${\underline{\nabla}}$ is written in canonical 
(Cartesian) coordinates).
\end{notation}
\begin{notation}
\label{N2} Twofold  
underlined capital letters 
represent
(n,n)-matrices. The so-called Frobenius inner product is: 
\begin{equation} \label{17..def:}
\hspace*{3cm}{\underline{\underline{A}}}:{\underline{\underline{B}}}:=
tr({\underline{\underline{A}}}\cdot{\underline{\underline{B}}}^{T})\,=\,
\sum_{j=1}^{n}\sum_{k=1}^{n}a_{j,k}b_{j,k},
\end{equation}
where $tr(.)$ denotes the trace of the (product-)matrix.
\end{notation}
\noindent
The {\bf Poincar\'e-(Friedrichs-)inequalities} ensure, that  the spaces ${\bb
  W}_{2}^{1}\hspace{-.62cm}{~}^{{~}^{{~}^{o}}}\hspace{.2cm}(.)$ and  
${\underline{\bb V}}(.)$ 
can be equipped with equivalent
norms generated by the Dirichlet norms 
\begin{align} \label{Dirichlet} 
\|u\|_{D}  :=
\Big(\sum_{k=1}^{n}\Big\|{\frac{\displaystyle{\partial
  u}}{\displaystyle{\partial x_{k}}}}\Big\|_{2}^{2} 
\Big)^{1/2}\quad \forall \, u\,\in\,{\bb
  W}_{2}^{1}\hspace{-.62cm}{~}^{{~}^{{~}^{o}}}\hspace{.2cm}(.)\,,\quad 
\|{\underline{u}}\|_{D,S}   := 
\Big(\sum_{j,k=1}^{n}
\Big\|
{\frac{\displaystyle{\partial u_{j}}}{\displaystyle{\partial x_{k}}}}
\Big\|_{2}^{2}
\Big)^{1/2}
\quad \forall \,  {\underline{u}}  \,\in\,{\underline{\bb V}}(.)\,,
\end{align}
where the Frobenius inner product is part of the last
definition. There we use ${\underline{\underline{A}}}\,=\,{\underline{\underline{B}}}\,:={\underline{\nabla}}\cdot
{\underline{u}}^T$.\\
\noindent
Denoting by $c_{p}({.})$ and $c_{p,S}({.})$ the Poincar\'e
constants with respect to the space ${\bb
  W}_{2}^{1}\hspace{-.62cm}{~}^{{~}^{{~}^{o}}}\hspace{.2cm}({\breve{\Omega}})$ 
and to the space ${\underline{\bb V}}({\breve{\Omega}})$,
respectively, the Poincar\'e-
(Friedrichs-)inequalities are 
\begin{align}\label{Poinc}
\|u\|_{2}  \leq   c_{p}({.})   \|u\|_{D} \quad\forall \, u\,
\in\,{\bb W}_{2}^{1}\hspace{-.62cm}{~}^{{~}^{{~}^{o}}}
\hspace{.2cm}({\breve{\Omega}})\,\quad
\mbox{and}\quad
\|{\underline{u}}\|_{{\underline{\bb L}}_{{\;\!}{2}}}
\leq c_{p,S}({.})
\|{\underline{u}}\|_{D,S} \quad
\forall \,  {\underline{u}}  \,\in\,
{\underline{\bb V}}({\breve{\Omega}})\,.
\end{align}
The Poincar\'e constants 
are related to the 
first eigenvalue of the Laplace or Stokes operator on the domains
${\breve{\Omega}}$
(with vanishing Dirichlet traces), respectively,
because of the relations (see the Ths. in Subsections 4.5.3 and 4.5.4 and Th. 3 in
6.1.5 in \cite{Triebel})
\begin{equation} \label{simple-tool}
c_{p}({.}) = 
\big(\lambda_{1,L}(.)\big)^{-1/2}
\quad
\,\quad \mbox{and}\quad
c_{p,S}({.}) =
\big(\lambda_{1,S}({.}) \big)^{-1/2} \,,
\end{equation}
where ${\lambda}_{1,L}({.})$ and  ${\lambda}_{1,S}(.)$
\begin{notation}\label{N3} In ${\bb
  R}^{n},\,n\geq 3$, let ${\underline{\mathfrak{e}}}_{j}\,:=\,
(\delta_{j,1},\delta_{j,2}, \dots, \delta_{j,n})^{T}$  ($j=1,2,\dots,n$, with
Kronecker's delta) be the unit vectors in Cartesian coordinates.
The polar coordinates are  $r$,  $\vartheta_{1}$, $\dots$ , $\vartheta_{n-2}$ and $\varphi$ with the corresponding unit vectors
${\underline{\mathfrak{e}}}_{{\;\!}r}$, ${\underline{\mathfrak{e}}}_{{\;\!}\vartheta_{1}}$, $\dots$, ${\underline{\mathfrak{e}}}_{{\;\!}\vartheta_{n-2}}$  and
${\underline{\mathfrak{e}}}_{{\;\!}\varphi}$.
If we denote by $\{{\underline{\mathfrak{e}}}_{r},
{\underline{\mathfrak{e}}}_{\vartheta_{1}}, \dots ,
{\underline{\mathfrak{e}}}_{\vartheta_{n-2}},{\underline{\mathfrak{e}}}_{\varphi}\}$
the system of these unit vectors in spherical polar coordinates then ${\underline{u}}$
is representable in both systems as\\[.1cm]
\hspace*{1.3cm}${\underline{u}}\,=\,\sum_{j=1}^{n} u_{j}{\underline{\mathfrak{e}}}_{j}\,=\, \sum_{j=1}^{n} u_{j,{\mathfrak{c}}}{\underline{\mathfrak{e}}}_{j}\,=\, 
 u_{r}{\underline{\mathfrak{e}}}_{r}+
 \sum_{k=1}^{n-2}
 u_{\vartheta_{k}}{\underline{\mathfrak{e}}}_{\vartheta_{k}}+
 u_{\varphi}{\underline{\mathfrak{e}}}_{\varphi}\,=\, 
 u_{r,{\mathfrak{s}}}{\underline{\mathfrak{e}}}_{r}+
\sum_{k=1}^{n-2}
 u_{\vartheta_{k},{\mathfrak{s}}}{\underline{\mathfrak{e}}}_{\vartheta_{k}}
+
 u_{\varphi,{\mathfrak{s}}}{\underline{\mathfrak{e}}}_{\varphi}$.\\[.1cm] 
The transformation from one coordinate system to the other is 
${\underline{u}}_{\mathfrak{c}}\,=\,
{\underline{\underline{T}}}_{{\mathfrak{c}},{\mathfrak{s}}}{\underline{u}}_{\mathfrak{s}}$
or ${\underline{u}}_{\mathfrak{s}}\,=\,
{\underline{\underline{T}}}^{-1}_{{\mathfrak{c}},{\mathfrak{s}}}
{\underline{u}}_{\mathfrak{c}}\,=\,
{\underline{\underline{T}}}_{{\mathfrak{s}},{\mathfrak{c}}}{\underline{u}}_{\mathfrak{c}}$, respectively
(these use columns of coordinates).
The transformation matrices ${\underline{\underline{T}}}_{{\mathfrak{c}},{\mathfrak{s}}}$ and ${\underline{\underline{T}}}_{{\mathfrak{s}},{\mathfrak{c}}}$ are given in the Appendix. 
\end{notation}
\noindent Our paper is organized as follows:
We collect the theoretical background in Section \ref{Sec2} at first. There we outline
the procedures to construct the Laplacian as well as the Stokes operator
as Friedrichs' extension from the Poisson and the Stokes
problem, respectively. We introduce the {\em Leray-Helmholtz projector} 
$\Upsilon : {\underline{\bb
    L}}_{{\;\!}2}(.)\,\longmapsto \,{\underline{\bb H}}(.)$ and we benefit from the properties of operators with a pure real
point spectrum as well as the criteria for the smallest eigenvalues
${\lambda}_{1,L}(.)$, ${\lambda}_{1,S}(.)$, 
and their corresponding eigenfunctions for the unit ball ${\Omega}_{(n)}^{*}$.
{{The results of our studies (cf. \cite{RuRuTh2016} -- \cite{RuRuTh2026}) with the special focus on the unit ball ${\Omega}_{(n)}^{*}$ are summarized in Subsection \ref{sec_ball_theo}. 
We refer to the appendix for the spherical surface harmonics of degree $\ell = 1$ and $\ell = 0$, the Laplacian in spherical coordinates as well as properties of the Laplace-Beltrami operator ${\boldsymbol B}$.
Our main results are formulated in Section \ref{Est_Poi_sec}.
There we derive the estimates of the Poincar\'e constants for $\Omega_{(n)}$.
The corresponding upper bounds 
are shown in Subsection \ref{Poin_up_subsec},
where we use the tools of Subsection \ref{sec_ball_theo} and the transformation $F(.)$ (cf. \eqref{transf}).
Finally, in Subsection \ref{sec_eval_results} we provide lower and
upper bounds of the Poincar\'e constants for $\Omega_{(n)}$. 
There we compare lower and upper bounds of the Poincar\'e constants
for $n-$dimensional cubes with the exact values 
for the evaluation of our results.
Additionally we list significant parameters (like volumes etc.) of the
balls and cubes in Subsection \ref{sec_eval_results}. }} 
\section{Theoretical groundwork}\label{Sec2}
\subsection{Laplace and Stokes operators on smooth bounded domains} \label{sec_theo} 
\begin{definition}\label{D4} {\em In Cartesian coordinates the Laplace operator is defined as
\begin{align*}
{\boldsymbol L^{\circledast}}\,{v} := -\Big(
{\displaystyle{\frac{{\partial}^{2} v}{\partial x_{1}^{2}}}}+{\displaystyle{\frac{{\partial}^{2} v}{\partial x_{2}^{2}}}}+\dots\,
{\displaystyle{\frac{{\partial}^{2} v}{\partial x_{n}^{2}}}}\Big)
=- \Delta{\;\!}_{\underline{x} } v
\hspace*{0.7cm}  \forall\,v\,\in\,D({\boldsymbol
  L^{\circledast}})=C_{o}^{\infty}({\breve{\Omega}})\,. 
\end{align*}
We denote Friedrichs' extension of ${\boldsymbol L^{\circledast}}$
by ${\boldsymbol 
  L}:={\overline{\boldsymbol L^{\circledast}}}$, where ${\boldsymbol
  L}$ is defined on 
$D({\boldsymbol L})\,:=\,{\bb
  W}_{2}^{1}\hspace{-.62cm}{~}^{{~}^{{~}^{o}}}\hspace{.2cm}({\breve{\Omega}}) 
\cup {\bb W}_{2}^{2}({\breve{\Omega}})$.}
\end{definition}
\begin{remark} The range of ${\boldsymbol L}$ is
$R({\boldsymbol L})={\bb L}_{2}({\breve{\Omega}})$. In this sense we may write:
${\boldsymbol L}=-\Delta{\;\!}_{\underline{x} } : D({\boldsymbol
  L})\,\longmapsto \,{\bb L}_{2}(.)$.
\end{remark}
\noindent We need the Leray-Helmholtz projection $\Upsilon$ to define the Stokes operator. 
$\Upsilon$ is the well-defined
projector of ${\underline{\bb L}}_{{\;\!}2}(.)$ onto its subspace 
${\underline{\bb H}}(.)$
of generalised solenoidal fields with vanishing generalised traces in
the normal direction. 
We note, that 
it is also used in the sense of:
$\Upsilon :{\underline{\bb W}}_{{\;\!}2}^{1}(.)\,\longmapsto
\,{\underline{\bb V}}(.)$\,.
\begin{definition}\label{D5} 
{\em The Stokes operator is defined as
$
{\boldsymbol S^{\circledast}}\,{\underline{v}}:=-
\Delta_{\underline{x} } {\underline{v}}\hspace*{0.5cm}
\forall\,{\underline{v}}\in D({\boldsymbol
  S^{\circledast}})=\underline{\cal V}({\breve{\Omega}})\, 
$, where ${\underline{v}}\,=\,{\underline{v}}_{{\;\!}\mathfrak{c}}$ is written in Cartesian coordinates
and the vector Laplace operator $\Delta_{\underline{x} }$ acts as a
scalar on each component. 
We denote Friedrichs' extension of ${\boldsymbol S^{\circledast}}$ by 
${\boldsymbol S}:={\overline{\boldsymbol S^{\circledast}}}$, where
${\boldsymbol S}$ is defined on its domain 
$D({\boldsymbol S}):={\underline{\bb S}}_{{\;\!}}^{2}(.)=
{\underline{\bb W}}_{{\;\!}2}^{2}(.)\cap{\underline{\bb V}}(.)$ \,.} 
\end{definition}
\begin{remark} The range of ${\boldsymbol S}$ is $R({\boldsymbol
  S})={\underline{\bb H}}(.) $. In this context one may write
${\boldsymbol S} =-\Upsilon \Delta{\;\!}_{\underline{x}
}:{\underline{\bb S}}_{{\;\!}}^{2}(.) 
\,\longmapsto \,{\underline{\bb H}}(.)$. One
can also use the Laplace operator in spherical polar
coordinates $\Delta_{r,\vartheta_{1},\dots\,\vartheta_{n-2},\varphi}$
(cf. Remark \ref{Lapldimn}). 
There one has to avoid to
choose spherical polar
coordinates for both, $\Delta$
and  ${\underline{v}}\,=\,{\underline{v}}_{{\;\!}\mathfrak{s}}$, since
the vector Laplacian when  
applied to ${\underline{v}}_{{\;\!}\mathfrak{s}}\,$ produces
convoluted tensor fields in this combination 
(see \cite{MoonSpencer}).
\end{remark}
\noindent We sketch the fundamental properties of both operators (i.e. ${\boldsymbol L}$
as well as ${\boldsymbol S}$) using ${\boldsymbol S}$ as an example.
\begin{theorem} \label{thmstok}
The Stokes operator ${\boldsymbol S}$ is positive and self-adjoint.
Its inverse ${\boldsymbol S}^{-1}$ is injective, self-adjoint and compact.
\end{theorem}
\noindent The proof of Theorem \ref{thmstok} is a simple modification of 
Theorems 4.3 and 4.4 in \cite{CoFoi}. The essential tools are
the Rellich theorem and the Lax-Milgram lemma.
The well-known theorem of Hilbert (see, e.g. \cite{CouHil}) and
regularity results like \cite[Prop.~I.2.2]{Temam}
lead to more precise results, namely:
\begin{corollary}
\label{STOeiFU}
The Stokes operator only has a point spectrum.
 All eigenvalues $\lambda_{j}$
of  ${\boldsymbol S}$ are real and of finite multiplicity.
The associated eigenfunctions 
$\{{\underline{w}}_{j}({\underline{x}})\}_{j=1}^{\infty}$
(counted in multiplicity)
are an orthogonal basis of
${\underline{\bb H}}(.)$ and ${\underline{\bb V}}(.)$, i.e.
\begin{align*}
{{(a)}}&\quad {\boldsymbol 
S}{\underline{w}}_{j}:=\lambda_{j}{\underline{w}}_{j}\quad\mbox{for
}\quad{\underline{w}}_{j}\in D({\boldsymbol S})
\quad\forall\,j\in \NN\\
{ (b)}& \quad
0\,<\lambda_{1}\leq\,\lambda_{2}\,\leq\cdots\leq\,\lambda_{j}
\,\leq\cdots\quad\mbox{and}\quad 
\lim_{j\rightarrow\infty}\lambda_{j}=\infty
\\[-2mm]
{(c)}&\quad
\|{\underline{w}}_{j}\|_{{\underline{\bb 
H}}}\,=1\quad\forall \, j\in \NN\,.
\end{align*}
\end{corollary} 
\noindent Concluding we write the eigenvalue problem for the Stokes operator ${\boldsymbol S}$ (cf. Definition \ref{D5}) 
on ${\breve{\Omega}}_{(n)}$ 
in a classical sense.
\begin{remark}\label{StokesClass}
We write the eigenvalue problem
for the Stokes operator ${\boldsymbol S}$ on ${\breve{\Omega}}_{(n)}$
in Cartesian coordinates in the classical sense. Namely we look for ${\underline{v}}\,\in\,{\underline{\bb C}}^2(.)\,$ and $\,  p  \in\,{\bb C}^1(.)$ fulfilling: 
\begin{align}  \label{Sto_Class}
\begin{array}{rcl}
-\,\Delta_{\underline{x}} \,{\underline{v}} \,+\,{\underline{\nabla}} \, p &\,=\,&\lambda {\underline{v}}\\[.2cm]
\mbox{div}{\;\!}\,{\underline{v}}\,=\,{\underline{\nabla}}^T{\underline{v}} &\,=\,& 0 \\[.2cm]
\mbox{with}\,\,{\underline{v}}_{|\partial{{\breve{\Omega}}_{(n)}}}&\,=\,&{\underline{0}}\,.
\end{array}
\end{align}
Under higher regularity assumptions the first two equations ensure that $p$
is a harmonic function, i.e. 
$\Delta_{\underline{x}} \,p\,=\,0\,$.
\end{remark}
\subsection{First respective Eigenvalues and Eigenfunction of Laplace and Stokes operators on n-dimensional balls} \label{sec_ball_theo} 
First let us cite significant results of \cite{RuRuTh2026} (Note 12, Corollary 20 and a part of Theorem 22) for 
the the open unit ball ${\Omega}_{(n)}^{*}$.
\begin{note} \label{Trieb_Laplace_1} \cite[\sc{Note} 12]{RuRuTh2026}\\
Let $n>1$ and ${\Omega}_{(n)}^{*}\subset {\bb R}^{n}$ be the open unit ball. We denote the Laplace 
operator on 
${\Omega}_{(n)}^{*}$ by ${\boldsymbol L}_{o}$.
The first simple eigenvalue
${\lambda}_{1,L}$ of 
${\boldsymbol L}_{o}$ on  ${\Omega}_{(n)}^{*}$ is the square of 
the smallest  positive solution
$\kappa_{1,L}$  of
\begin{align}\label{Lapl_einheitskugel_EW}
J_{\frac{n}{2}-1}(\kappa_{L}) =0\quad \quad \quad \quad \,\forall \,n\,\in\,{\bb N}:\,n>1\,.
\end{align}
The corresponding first eigenfunctions are
\begin{align}\label{Lapl_einheitskugel_EF}
{{w}}_{{\;\!}1,L}({\underline{x}})
\,=\,\frac{\tilde{c}_{1,L}}{r^{\frac{n}{2}-1}}
\cdot
J_{\frac{n}{2}-1}(\kappa_{1,L}\cdot r)\, ,
\end{align}
where,
$\tilde{c}_{1,L}$ denotes a 
scaling constant. 
The functions in \eqref{Lapl_einheitskugel_EF}  only 
depend on $r\,\in\,(0,1)$ and are well-defined in $r=0$ as well. The corresponding spherical harmonic
functions of degree $\ell = 0$ are $S^{\{0\}}=1$, cf. \eqref{SphSurf0}.
\end{note}
\begin{corollary} \label{Stokes_Unit_ball} ({\sc Corollary 20} \cite{RuRuTh2026})\\
Let ${\bb R}^{n}, n>1$, and ${\Omega}_{(n)}^{*}\subset {\bb R}^{n}$ be the open unit ball. We denote the Stokes 
operator on 
${\Omega}_{(n)}^{*}$ by ${\boldsymbol S}_{o}$.
For any $n\,\in\,{\bb N},\,n>1$, the first simple eigenvalue
${\lambda}_{1,S}$ of  
${\boldsymbol S}_{o}$ 
is the square of 
the smallest  positive solution
$\kappa_{1,S}$  of
\begin{align}\label{Stokes_einheitskugel_EW}
J_{\frac{n}{2}}(\kappa_{S}) =0\quad \quad \quad \quad \,\forall \,n>1\,.
\end{align}
For any $n>1$ one of the corresponding first eigenfunctions are the vector fields
\begin{align}\label{Stokes_Allgem_Einsball_EF}
{\underline{w}}_{{\;\!}\varphi,S}({\underline{x}})\,:=\,{\tilde{c}_{1,S}}
\left\{
\begin{array}{ll}
\, \varrho(r) \sin{\vartheta_{1}}\sin{\vartheta_{2}}\cdots \sin{\vartheta_{n-2}}\cdot 
{\underline{\mathfrak{e}}}_{{\;\!}\varphi}\,& \forall \,n\,\in\,{\bb N}\,;\,n>2  \\
{~} & {~} \\ 
\, \varrho(r) \cdot 
{\underline{\mathfrak{e}}}_{{\;\!}\varphi}\,& 
\text{for }\,
\,n=2 \, ,
\end{array}\right.
\end{align}
where $\tilde{c}_{1,S}$ denotes a 
scaling constants. 
The $\varrho(r)$ are scalar functions of \,$r\,\in\,(0,1)$ defined as
\begin{align}\label{Stokes_Allgem_Rho_Einsball}
\varrho(r)
\,:=\,\frac{1}{r^{\frac{n}{2}-1}}
J_{\frac{n}{2}}(\kappa_{1,S} \,r)
&\,\quad\quad\forall \,
n\,\in\,{\bb N},\,n >1. 
\end{align}
Especially one has \quad ${\boldsymbol S}_{o}\,{\underline{w}}_{{\;\!}\varphi,S}\,=\,
{\lambda}_{1,S}\,{\underline{w}}_{{\;\!}\varphi,S}\,=\,(\kappa_{1,S})^2\,{\underline{w}}_{{\;\!}\varphi,S}\,.$
All is well-defined in $r=0$ as well.
\end{corollary}
\noindent In what follows we use the dimension $n$ of the spaces ${\bb R}^{n}$ as an index in the notation
of the roots, of the eigenvalues and of the Poincar\'e constants too.
\begin{theorem} \label{thm_shift_Lapl_Stokes_n} ({\sc Theorem 22} \cite{RuRuTh2026})\\
Let $n >1$  and consider ${\Omega}_{(n)}^{*}\subset {\bb R}^{n}$ resp.  
${\Omega}_{(n+2)}^{*}\subset {\bb R}^{n+2}$.
The first eigenvalues ${\lambda}_{1,S,(n)}(o)\,=\,{\kappa^2_{1,S,(n)}(o)}$ of the Stokes operator 
on ${\Omega}_{(n)}^{*}$ in  ${\bb
  R}^{n}$ coincide with the first eigenvalues 
resp. ${\lambda}_{1,L,(n+2)}(o)\,=\,{\kappa^2_{1,L,(n+2)}(o)}$ of the Laplacian
on on ${\Omega}_{(n+2)}^{*}$ in  ${\bb R}^{n+2}$.  
Due to  \eqref{Stokes_einheitskugel_EW} and 
\eqref{Lapl_einheitskugel_EW} we see that \\[.25cm]
\hspace*{5cm}
$J_{\frac{n}{2}}(\kappa_{1,S,(n)})\,=\,J_{\frac{n+2}{2}-1}(\kappa_{1,L,(n+2)})\,=\,0\,.$\\[.1cm]
In particular,  
for the Poincar\'{e} constants it holds that $c_{p,S,(n)}({\Omega}_{(n)}^{*})\,=\,c_{p,(n+2)}({\Omega}_{(n+2)}^{*})$
$\forall\,n \,\in\,{\bb N} ,\,n >1$.
\end{theorem} 
\begin{remark} \label{Besselzeros}
The values of $\kappa_{1,L,(n)}(o)$ for ${\Omega}_{(n)}^{*}$ are the
first roots of the corresponding Besselfunctions. 
In ascending dimensional order
 $n\,= \,2, \,3 ,\,4 ,\,5,\,6,\,7$ 
  we find $\kappa_{1,L,(2)}({o})\,\approx\, 2.404825558 $,\,\, 
$\kappa_{1,L,(3)}({o})\,=\,\pi $,\,\,$\kappa_{1,L,(4)} ({o})\,\approx\, 3.831705970 $,\,\,
$\kappa_{1,L,(5)}({o})\,\approx\,4.493409458 $,\,\,$\kappa_{1,L,(6)}({o})\,\approx\, 5.135622302 $,
$\kappa_{1,L,(7)}({o})\,\approx\,5.763459197 $,\,\,$\dots$\,\,.
\end{remark}
\noindent Now we are going to use the transformation $F(.)$ of the closed ball ${\overline{\tilde{\Omega}^{{\!{~}}^{\scriptstyle{*}}}}
}_{\hspace*{-.1cm}\!\tilde{\underline{x}},R,(n)}$ onto the closed unit ball 
${\overline{\White{\hat\Omega}^{{\!{~}}^{\scriptstyle{*}}}\hspace*{-.45cm}
\Omega}}_{(n)}$ (cf. \eqref{transf}) 
for the determination of the eigenvalues of ${\boldsymbol L}$ and ${\boldsymbol S}$ on ${{\tilde{\Omega}^{{\!{~}}^{\scriptstyle{*}}}}
}_{\hspace*{.1cm}\!\tilde{\underline{x}},R,(n)}$ with the special view on the 
first eigenvalues of ${\boldsymbol L}$ and ${\boldsymbol S}$.
\begin{remark} \label{transformationR} 
The radius  $R\,=\, \frac{1}{2}\text{diam}\,{\Omega}_{(n)}\,$ of \,$\,{{\tilde{\Omega}^{{\!{~}}^{\scriptstyle{*}}}}
}_{\hspace*{.1cm}\!\tilde{\underline{x}},R,(n)}$ is the significant 
weight or scaling factor by use of the transformation $F(.)$. Let us denote by 
$\{\lambda_{j,L,(n)}(o)\}_{j=1}^{\infty}$ and by 
$\{\lambda_{j,L,(n)}({{\tilde{\Omega}^{{\!{~}}^{\scriptstyle{*}}}}
}_{\hspace*{.1cm}\!\tilde{\underline{x}},R,(n)})\}_{j=1}^{\infty}$ the monotonically increasing sequences of the eigenvalues for ${\boldsymbol L}_{o}$ and ${\boldsymbol L}({{\tilde{\Omega}^{{\!{~}}^{\scriptstyle{*}}}}
}_{\hspace*{.1cm}\!\tilde{\underline{x}},R,(n)})$ counted in multiplicity. 
We apply the notations
$\{\lambda_{j,S,(n)}(o)\}_{j=1}^{\infty}$ and by 
$\{\lambda_{j,S,(n)}({{\tilde{\Omega}^{{\!{~}}^{\scriptstyle{*}}}}
}_{\hspace*{.1cm}\!\tilde{\underline{x}},R,(n)})\}_{j=1}^{\infty}$ the monotonically increasing sequences of the eigenvalues for ${\boldsymbol S}_{o}$ and ${\boldsymbol S}({{\tilde{\Omega}^{{\!{~}}^{\scriptstyle{*}}}}
}_{\hspace*{.1cm}\!\tilde{\underline{x}},R,(n)})$ counted in
multiplicity again. 
For all the eigenvalues a simple calculation with the the
transformation $F(.)$ leads to
\begin{align}\label{TransfEigenv} 
{\lambda_{j,L,(n)}}(o)\,=\, R^2\cdot {\lambda_{j,L,(n)}({{\tilde{\Omega}^{{\!{~}}^{\scriptstyle{*}}}}
}_{\hspace*{.1cm}\!\tilde{\underline{x}},R,(n)})} \,\,,\quad\, &\quad \nonumber\\[-.2cm]
{~} &\quad\forall \, j\in \NN\,.\\[-.2cm]
{\lambda_{j,S,(n)}}(o)\,=\, R^2\cdot {\lambda_{j,S,(n)}({{\tilde{\Omega}^{{\!{~}}^{\scriptstyle{*}}}}
}_{\hspace*{.1cm}\!\tilde{\underline{x}},R,(n)})} \,\,,\quad\, &\quad \nonumber
\end{align}
\end{remark}
\begin{note}\label{first_eigenv_R}
Let us 
take the open unit ball ${\Omega}_{(n)}^{*}\subset {\bb R}^{n}$ for $n\,\in\, {\bb N},\,1\,<\,n\,<\,\infty$.
We realize that the first simple eigenvalues of the Laplace 
operator ${\boldsymbol L}_{o}$ and the first eigenvalues of the 
Stokes operator ${\boldsymbol S}_{o}$ in ${\Omega}_{(n)}^{*}$ are the squares of 
the smallest positive solution $\kappa_{1,L,(n)}(o)$ and $\kappa_{1,S,(n)}(o)$ of
\begin{align}\label{Gucken}
J_{\frac{n}{2}-1}(\kappa_{L}) =0\quad \quad {\text{and}} \quad \quad J_{\frac{n}{2}}(\kappa_{S}) =0\,
\quad\quad \quad \,,\,\,\forall \,n\,\in\, {\bb N},\,1\,<\,n\,<\,\infty.
\end{align}
Using the radius  $R\,=\, \frac{1}{2}\text{diam}\,{\Omega}_{(n)}\,$ of \,$\,{{\tilde{\Omega}^{{\!{~}}^{\scriptstyle{*}}}}
}_{\hspace*{.1cm}\!\tilde{\underline{x}},R,(n)}$ we find $\forall \,
n\in \NN\,,\,1\,<\,n\,<\,\infty$ that
\begin{align}\label{Transf_1st_Eigenv} 
\left({\kappa_{1,L,(n)}}(o)\right)^2\,=\,
{\lambda_{1,L,(n)}}(o)\,=\, R^2{\lambda_{1,L,(n)}({{\tilde{\Omega}^{{\!{~}}^{\scriptstyle{*}}}}
}_{\hspace*{.1cm}\!\tilde{\underline{x}},R,(n)})} 
, \quad
  \left({\kappa_{1,S,(n)}}(o)\right)^2\,=\,
{\lambda_{1,S,(n)}}(o)\,=\, R^2{\lambda_{1,S,(n)}({{\tilde{\Omega}^{{\!{~}}^{\scriptstyle{*}}}}
}_{\hspace*{.1cm}\!\tilde{\underline{x}},R,(n)})} \,.
\end{align}
\end{note}
\section{Estimates of the Poincar\'e Constants \label{Est_Poi_sec}}
Subsequently we will study the monotonic behavior of the first simple eigenvalues 
${\lambda_{1,L,(n)}({{\tilde{\Omega}^{{\!{~}}^{\scriptstyle{*}}}}
}_{\hspace*{.1cm}\!\tilde{\underline{x}},R,(n)})}$ of the Laplace 
operator ${\boldsymbol L}$ and the first eigenvalues ${\lambda_{1,S,(n)}({{\tilde{\Omega}^{{\!{~}}^{\scriptstyle{*}}}}
}_{\hspace*{.1cm}\!\tilde{\underline{x}},R,(n)})}$
of the 
Stokes operator ${\boldsymbol S}$ on $\,{{\tilde{\Omega}^{{\!{~}}^{\scriptstyle{*}}}}
}_{\hspace*{.1cm}\!\tilde{\underline{x}},R,(n)}$ in relation to the first simple eigenvalues 
${\lambda_{1,L,(n)}({\Omega}_{(n)})}$ of the Laplace 
operator ${\boldsymbol L}$ and the first eigenvalues ${\lambda_{1,S,(n)}({\Omega}_{(n)})}$ 
of the 
Stokes operator ${\boldsymbol S}$ on $\,{\Omega}_{(n)}$, where in the
definition of the radius  $R\,=\,
\frac{1}{2}\text{diam}\,{\Omega}_{(n)}\,$ we suppose that $\,{\Omega}_{(n)}$ is our simply connected 
and convex ${\cal{C}}^{2}$-domain. 
In what follows we are going to use Courants max-min principle \cite{CouMaxmin} for linear selfadjoint operators on 
Hilbert spaces (with the well known counterpart of the minimax
principle for linear compact selfadjoint operators on  
Hilbert spaces as the view on inverse operators cf. \cite{Poinc}). We
are going to formulate the variation principle for the eigenvalues 
directly without applying the standard tool \cite[Satz 3 on page 355 in Bd. 1]{CouHil}. 
\subsection{Upper Bound Estimates of the Poincar\'e Constants for $\,{\Omega}_{(n)}$} \label{Poin_up_subsec}
It is our first step to
look for the first eigenvalues of the 
operators ${\boldsymbol L}_{o}$ and the first eigenvalues of the 
Stokes operator ${\boldsymbol S}_{o}$ on ${\Omega}_{(n)}^{*}$ to compare these eigenvalues with 
the corresponding first eigenvalues
of the Laplacian  ${\boldsymbol L}(.)$ and of the Stokes operator ${\boldsymbol S}(.)$ on ${\hat{\Omega}}_{(n)}$.
We start with the formulation of standard variation principles, 
\cite{CouHil}.
\begin{prop} \label{EW_Laplace_1} The Laplacians  ${\boldsymbol L}_{o}$ and ${\boldsymbol L}({\hat{\Omega}}_{(n)})$ on ${\Omega}_{(n)}^{*}$ and \, ${\hat{\Omega}}_{(n)}$ are selfadjoint and positive operators with a pure point spectrum
for every fixed
$n\,\in\,{\bb N}:\,n>1$.
The first eigenvalue
$\lambda_{1,L,(n)}(o)$ and  first eigenvalue $\lambda_{1,L,(n)}({\hat{\Omega}}_{(n)})$ are characterized
by the Dirichlet norm c.f. \eqref{Dirichlet}:
\begin{align} \label{n_Lapl_Einskug}
\lambda_{1,L,(n)}({o}) & = \,\min_{u \in {\bb W}_{2}^{1}\hspace{-.62cm}{~}^{{~}^{{~}^{{~}^{o}}}}\hspace{.1cm}({\Omega}_{(n)}^{*}): \| u \|_{{\bb L}_{2}(.)}=1 }{\;\!}{\|u\|_{D}^{\,2} } \,=\,
     \min_{u \in  {\bb W}_{2}^{1}\hspace{-.62cm}{~}^{{~}^{{~}^{{~}^{o}}}}\hspace{.1cm}
    ({\Omega}_{(n)}^{*}): 
\| u \|_{{\bb L}_{2}(.)}=1 }{\;\!}
    \int_{{\Omega}^{*}_{(n){~}}} 
    ({\underline{\nabla}} {\;\!} u({\underline{x}}))^T\cdot ({\underline{\nabla}} {\;\!} u({\underline{x}}))\
   d{\underline{x}} \, \quad  \nonumber\\
{~}  & \hspace*{12.62cm} \,\text{and} \\
\lambda_{1,L,(n)}({\hat{\Omega}}_{(n)}) & = \,
 \min_{u \in {\bb W}_{2}^{1}\hspace{-.62cm}{~}^{{~}^{{~}^{{~}^{o}}}}\hspace{.1cm}
    (\hat{\Omega}_{(n)}):
\| u \|_{{\bb L}_{2}(.)}=1 }{\;\!}
{\|u\|_{D}^{\,2} } \,=\,
     \min_{u \in {\bb W}_{2}^{1}\hspace{-.62cm}{~}^{{~}^{{~}^{{~}^{o}}}}\hspace{.1cm}
    (\hat{\Omega}_{(n)}):
\| u \|_{{\bb L}_{2}(.)}=1 }{\;\!}
    \int_{{\hat{\Omega}}_{(n)~}} 
    ({\underline{\nabla}} {\;\!}u({\underline{x}}))^T\cdot ({\underline{\nabla}} {\;\!}u({\underline{x}}))\
   d{\underline{x}} \, \quad .\nonumber
\end{align}  
For the eigenvalues we have the inequality
\begin{align}\label{mon_Lapl_Einskug}
\lambda_{1,L,(n)}({\hat{\Omega}}_{(n)})\,\geq \,\lambda_{1,L,(n)}({o})\,.
\end{align}
This relation is due to 
${\hat{\Omega}}_{(n)}\,\subset 
{\Omega}^{*}_{(n)}$. 
\end{prop}
\begin{proof} For the properties of the Laplacians  ${\boldsymbol
    L}_{o}$ and ${\boldsymbol L}({\hat{\Omega}}_{(n)})$ 
we cite the Theorem from Subsection 4.5.1 and Remark 4 from Subsection 6.2.2 in
\cite{Triebel}. The relations in \eqref{n_Lapl_Einskug} are a simple
consequence of Courant's max-min principle. 
The eigenvalues of ${\boldsymbol L}^{-1}_{o}$ and ${\boldsymbol
  L}^{-1}({\hat{\Omega}}_{(n)})$
are defined by the Minimax-principle cf. the Theorems from Subsection 4.1.7 and
Subsection 4.2.5 in \cite{Triebel}.
We introduce the spaces
\begin{align} \label{subsetsL}
{\bb L}_{2}({\Omega}^{*}_{{(n)}},0)\,:=\,&
\{v\,\in\,{\bb L}_{2}({\Omega}^{*}_{(n)}): v=0\,\, {\mbox{a.e. on }} {\Omega}^{*}_{(n)}
\setminus {\hat{\Omega}_{(n)}}\}\,\subset\,{\bb L}_{2}({\Omega}^{*}_{{(n)}})\,\,\,\,\quad\,\text{and}\,\nonumber\\[-.2cm]
{~} & {~}\\[-.2cm]
{\bb W}_{2}^{1}\hspace{-.62cm}{~}^{{~}^{{~}^{o}}}\hspace{.2cm}({\Omega}^{*}_{{(n)}},0)\,:=\,&
\{v\,\in\,{\bb W}_{2}^{1}\hspace{-.62cm}{~}^{{~}^{{~}^{o}}}\hspace{.2cm}({\Omega}^{*}_{{(n)}}): 
v_{|\hat{\Omega}_{(n)}}\,\in\,{\bb W}_{2}^{1}\hspace{-.62cm}{~}^{{~}^{{~}^{o}}}\hspace{.2cm}(\hat{\Omega}_{{(n)}})
\,\,\text{and}\,
v=0\,\, {\mbox{a.e. on }} {\Omega}^{*}_{(n)}
\setminus {\hat{\Omega}_{(n)}}\}\,\subset\,
{\bb W}_{2}^{1}\hspace{-.62cm}{~}^{{~}^{{~}^{o}}}\hspace{.2cm}({\Omega}^{*}_{{(n)}})\,\,.\nonumber
\end{align}
In view of \eqref{n_Lapl_Einskug} we observe that ${\bb W}_{2}^{1}\hspace{-.62cm}{~}^{{~}^{{~}^{o}}}\hspace{.2cm}
({\Omega}^{*}_{(n)},0))\,\cong \,{\bb W}_{2}^{1}\hspace{-.62cm}{~}^{{~}^{{~}^{o}}}\hspace{.2cm}(\hat{\Omega}_{(n)})\,$\,(
and ${\bb L}_{2}({\Omega}^{*}_{{(n)}},0)\,\cong \,{\bb L}_{2}(\hat{\Omega}_{{(n)}})$\,).
The proof of the inequality \eqref{mon_Lapl_Einskug} is a simple argument on extrema about real subsets.
For ${\Omega}^{*}_{(n)}\,\not=\,\hat{\Omega}_{{(n)}}$ with
$\hat{\Omega}_{{(n)}}\,\subset\,{\Omega}^{*}_{(n)}$ we find that:
\begin{align}\label{proof_mon_Lapl_Einskug}
\lambda_{1,L,(n)}({\hat{\Omega}}_{(n)})&\,=\,
\min_{u \in {\bb W}_{2}^{1}\hspace{-.62cm}{~}^{{~}^{{~}^{{~}^{o}}}}\hspace{.1cm}
(\hat{\Omega}_{(n)}):
\| u \|_{{\bb L}_{2}(.)}=1 }{\;\!}
    \int_{{\hat{\Omega}}_{(n)~}} 
    ({\underline{\nabla}} {\;\!}u({\underline{x}}))^T\cdot ({\underline{\nabla}} {\;\!}u({\underline{x}}))\
   d{\underline{x}} \,=\,\nonumber\\
{~}& \,=\,\min_{u \in {\bb W}_{2}^{1}\hspace{-.62cm}{~}^{{~}^{{~}^{{~}^{o}}}}\hspace{.1cm}
({\Omega}^{*}_{{(n)}},0):
\| u \|_{{\bb L}_{2}(.)}=1 }{\;\!}
    \int_{{{\Omega}}^*_{(n)~}} 
    ({\underline{\nabla}} {\;\!}u({\underline{x}}))^T\cdot ({\underline{\nabla}} {\;\!}u({\underline{x}}))\
   d{\underline{x}} \,\geq\, \\
{~}& \,\geq \,
\min_{u \in {\bb W}_{2}^{1}\hspace{-.62cm}{~}^{{~}^{{~},^{{~}^{o}}}}\hspace{.1cm}
    ({\Omega}_{(n)}^{*}): 
\| u \|_{{\bb L}_{2}(.)}=1 }{\;\!}
    \int_{{\Omega}^{*}_{(n){~}}} 
    ({\underline{\nabla}} {\;\!} u({\underline{x}}))^T\cdot ({\underline{\nabla}} {\;\!} u({\underline{x}}))\
   d{\underline{x}}\,\,= \,\,
   \lambda_{1,L,(n)}({o})\,\quad\,.\nonumber
\end{align}
\end{proof}
\noindent For the next Propositiown we refer again to \cite{CouHil}. 
\begin{prop} \label{EW_Laplace_2} The Laplacians  ${\boldsymbol L}({{\tilde{\Omega}^{{\!{~}}^{\scriptstyle{*}}}}
}_{\hspace*{.1cm}\!\tilde{\underline{x}},R,(n)})$ and ${\boldsymbol L}({{\Omega}}_{(n)})$ on ${{\tilde{\Omega}^{{\!{~}}^{\scriptstyle{*}}}}
}_{\hspace*{.1cm}\!\tilde{\underline{x}},R,(n)}$ and \, ${{\Omega}}_{(n)}$ are selfadjoint and positive operators with a pure point spectrum
for every fixed $n\,\in\,{\bb N}:\,n>1$.
The first eigenvalue
$\lambda_{1,L,(n)}({{\tilde{\Omega}^{{\!{~}}^{\scriptstyle{*}}}}
}_{\hspace*{.1cm}\!\tilde{\underline{x}},R,(n)})$ and  first eigenvalue $\lambda_{1,L,(n)}({{\Omega}}_{(n)})$ are characterized
by the Dirichlet norm c.f. \eqref{Dirichlet}:
\begin{align} \label{n_Lapl_DiaKug}
\lambda_{1,L,(n)}({{\tilde{\Omega}^{{\!{~}}^{\scriptstyle{*}}}}
}_{\hspace*{.1cm}\!\tilde{\underline{x}},R,(n)}) & = \,\min_{u \in {\bb W}_{2}^{1}\hspace{-.62cm}{~}^{{~}^{{~}^{{~}^{o}}}}\hspace{.1cm}({{\tilde{\Omega}^{{\!{~}}^{\scriptstyle{*}}}}
}_{\hspace*{.1cm}\!\tilde{\underline{x}},R,(n)}): \| u \|_{{\bb L}_{2}(.)}=1 }{\;\!}{\|u\|_{D}^{\,2} } \,=\,
     \min_{u \in  {\bb W}_{2}^{1}\hspace{-.62cm}{~}^{{~}^{{~}^{{~}^{o}}}}\hspace{.1cm}
    ({{\tilde{\Omega}^{{\!{~}}^{\scriptstyle{*}}}}
}_{\hspace*{.1cm}\!\tilde{\underline{x}},R,(n)}): 
\| u \|_{{\bb L}_{2}(.)}=1 }{\;\!}
    \int_{{{\tilde{\Omega}^{{\!{~}}^{\scriptstyle{*}}}}
}_{\hspace*{.1cm}\!\tilde{\underline{x}},R,(n){~}}} 
    ({\underline{\nabla}} {\;\!} u({\underline{x}}))^T\cdot ({\underline{\nabla}} {\;\!} u({\underline{x}}))\
   d{\underline{x}} \, \quad  \nonumber\\
{~}  & \hspace*{12.62cm} \,\text{and} \\
\lambda_{1,L,(n)}({{\Omega}}_{(n)}) & = \,
 \min_{u \in {\bb W}_{2}^{1}\hspace{-.62cm}{~}^{{~}^{{~}^{{~}^{o}}}}\hspace{.1cm}
    ({\Omega}_{(n)}):
\| u \|_{{\bb L}_{2}(.)}=1 }{\;\!}
{\|u\|_{D}^{\,2} } \,=\,
     \min_{u \in {\bb W}_{2}^{1}\hspace{-.62cm}{~}^{{~}^{{~}^{{~}^{o}}}}\hspace{.1cm}
    ({\Omega}_{(n)}):
\| u \|_{{\bb L}_{2}(.)}=1 }{\;\!}
    \int_{{{\Omega}}_{(n)~}} 
    ({\underline{\nabla}} {\;\!}u({\underline{x}}))^T\cdot ({\underline{\nabla}} {\;\!}u({\underline{x}}))\
   d{\underline{x}} \, \quad .\nonumber
\end{align}  
Using the inclusion ${{\Omega}}_{(n)}\,\subset \,{{\tilde{\Omega}^{{\!{~}}^{\scriptstyle{*}}}}
}_{\hspace*{.1cm}\!\tilde{\underline{x}},R,(n)}$ one gets the inequality for
the first eigenvalues:
\begin{align}\label{mon_Lapl_Diakug}
\lambda_{1,L,(n)}({{\Omega}}_{(n)})\,\geq \,\lambda_{1,L,(n)}({{\tilde{\Omega}^{{\!{~}}^{\scriptstyle{*}}}}
}_{\hspace*{.1cm}\!\tilde{\underline{x}},R,(n)})\,.
\end{align}
\end{prop}
\begin{proof} We follow the Proof of Proposition \ref{EW_Laplace_1}. One can repeat the steps of the Proof in general.
The spaces for the validation of the inequality for
the first eigenvalues \eqref{mon_Lapl_Diakug} are:
\begin{align} \label{subsetsL}
{\bb L}_{2}({{\tilde{\Omega}^{{\!{~}}^{\scriptstyle{*}}}}
}_{\hspace*{.1cm}\!\tilde{\underline{x}},R,(n)},0)\,:=\,&
\{v\,\in\,{\bb L}_{2}({{\tilde{\Omega}^{{\!{~}}^{\scriptstyle{*}}}}
}_{\hspace*{.1cm}\!\tilde{\underline{x}},R,(n)}): v=0\,\, \,{\mbox{a.e. on }} {{\tilde{\Omega}^{{\!{~}}^{\scriptstyle{*}}}}
}_{\hspace*{.1cm}\!\tilde{\underline{x}},R,(n)}
\setminus {{\Omega}_{(n)}}\}\,\subset\,{\bb L}_{2}({{\tilde{\Omega}^{{\!{~}}^{\scriptstyle{*}}}}
}_{\hspace*{.1cm}\!\tilde{\underline{x}},R,(n)})\,\,\,\,\quad\,\text{and}\,\nonumber\\[-.2cm]
{~} & {~}\\[-.2cm]
{\bb W}_{2}^{1}\hspace{-.62cm}{~}^{{~}^{{~}^{o}}}\hspace{.2cm}({{\tilde{\Omega}^{{\!{~}}^{\scriptstyle{*}}}}
}_{\hspace*{.1cm}\!\tilde{\underline{x}},R,(n)},0)\,:=\,&
\{v\,\in\,{\bb W}_{{\text{resp.}}2}^{1}\hspace{-.62cm}{~}^{{~}^{{~}^{o}}}\hspace{.2cm}({{\tilde{\Omega}^{{\!{~}}^{\scriptstyle{*}}}}
}_{\hspace*{.1cm}\!\tilde{\underline{x}},R,(n)}): 
v_{|{\Omega}_{(n)}}\,\in\,{\bb W}_{2}^{1}\hspace{-.62cm}{~}^{{~}^{{~}^{o}}}\hspace{.2cm}({\Omega}_{{(n)}})
\,\,\text{and}\,
v=0\,\, {\mbox{a.e. on }} {{\tilde{\Omega}^{{\!{~}}^{\scriptstyle{*}}}}
}_{\hspace*{.1cm}\!\tilde{\underline{x}},R,(n)}
\setminus {{\Omega}_{(n)}}\}\,\subset\,
{\bb W}_{2}^{1}\hspace{-.62cm}{~}^{{~}^{{~}^{o}}}\hspace{.2cm}({{\tilde{\Omega}^{{\!{~}}^{\scriptstyle{*}}}}
}_{\hspace*{.1cm}\!\tilde{\underline{x}},R,(n)})\,\,.\nonumber
\end{align}
The proof of the inequality \eqref{mon_Lapl_Diakug} 
is a simple argument on extrema about real subsets again. Finally one repeats the steps similar to \eqref{proof_mon_Lapl_Einskug}.
\end{proof} 
\noindent  In view of Note \ref{first_eigenv_R} we now look for similar properties of the first 
Stokes eigenvalues. The application of \cite[Satz 3 on page 355 in Bd. 1]{CouHil} and of the results
of \cite{Weid} is not possible here, but the direct use of Courant's
max-min principle works very well.
Additionally we use Corollary \ref{STOeiFU} for the properties of the Stokes operator.
\begin{prop} \label{EW_Stokes_12} 
Let ${\breve{\Omega}}$ stand as shorthand for any of the domains 
$\Omega_{(n)}$, $\tilde{\Omega}^*_{\,\tilde{\underline{x}},R,(n)}$, 
${\hat{\Omega}}_{(n)}$ and 
${\Omega}^*_{(n)}$ like in the {\bf General notation B}. The Stokes operators ${\boldsymbol S}({\breve{\Omega}}_{(n)})$ on ${\breve{\Omega}}_{(n)}$ are selfadjoint and positive operators with a pure point spectrum
for every fixed
$n\,\in\,{\bb N}:\,n>1$.
The first eigenvalue $\lambda_{1,S,(n)}({\breve{\Omega}}_{(n)})$ for each domain is characterized
by the Dirichlet norm:
\begin{align} \label{n_Stokes_alle}
\lambda_{1,S,(n)}({\breve{\Omega}}_{(n)}) & = \,
 \min_{\underline{u} \in {\underline{\bb V}}\hspace{.0cm}
    (\breve{\Omega}_{(n)}):
\| \underline{u}  \|_{{\underline{\bb L}_{2}(.)}}=1 }{\;\!}
{\|\underline{u} \|_{D,S}^{\,2} } \,=\, \min_{\underline{u} \in {\underline{\bb V}}\hspace{.0cm}
    (\breve{\Omega}_{(n)}):
\| \underline{u}  \|_{{\underline{\bb L}_{2}(.)}}=1 }{\;\!}
    \int_{{\breve{\Omega}}_{(n)~}} \big(
    \sum_{j=1}^{n}\sum_{j=1}^{n}\,\left(\frac{\partial u_j}{\partial x_k}\right)^2\big)
   d{\underline{x}}  \quad ,
\end{align} 
where the Frobenius inner product is part of the definition in the Dirichlet norm 
\eqref{Dirichlet}. 
For the eigenvalues we observe
\begin{align}\label{mon_Stokes__12}
\lambda_{1,S,(n)}({\hat{\Omega}}_{(n)})\,\geq \,\lambda_{1,S,(n)}({o})\,\quad\,{\text{and}}\,\quad\,
\lambda_{1,S,(n)}({{\Omega}}_{(n)})\,\geq \,\lambda_{1,S,(n)}({{\tilde{\Omega}^{{\!{~}}^{\scriptstyle{*}}}}
}_{\hspace*{.1cm}\!\tilde{\underline{x}},R,(n)})
\end{align}
due to  
${\hat{\Omega}}_{(n)}\,\subset\, 
{\Omega}^{*}_{(n)}$\,\,respectively\,\,$\Omega_{(n)}\,\subset\, \tilde{\Omega}^*_{\,\tilde{\underline{x}},R,(n)}$. 
Again we use the notation $\lambda_{1,S,(n)}({o})$ for $\lambda_{1,S,(n)}({\Omega}^*_{(n)})$
in \eqref{mon_Stokes__12}.
\end{prop}
\begin{proof} We note, that \eqref{n_Stokes_alle} is a simple
  consequence of Courant’s max-min principle.
We introduce the spaces
\begin{align} \label{subsetsSEinskugel}
{\underline{\bb L}}_{{\;\!}{2}}
({\Omega}^{*}_{(n)},\underline{0})\,:=\,&
\{{\underline{v}}\,\in\,{\underline{\bb L}}_{{\;\!}{2}}({\Omega}^{*}_{(n)}): \underline{v} =\underline{0} \,\, \,{\mbox{a.e. on }} {\Omega}^{*}_{(n)}
\setminus {\hat{\Omega}_{(n)}}\}\,\subset\,{\underline{\bb L}}_{{\;\!}{2}}({\Omega}^{*}_{(n)})\,\,\,\,\quad\,\text{and}\,\nonumber\\[-.2cm]
{~} & {~}\\[-.2cm]
{\underline{\bb V}} \hspace{.1cm}({\Omega}^{*}_{(n)},\underline{0})\,:=\,&
\{{\underline{v}}\,\in\,{\underline{\bb V}}\hspace{.1cm}({\Omega}^{*}_{(n)}): 
{\underline{v}}_{|{\hat\Omega}_{(n)}}\,\in\,{\underline{\bb V}}
\hspace{.1cm}({\hat\Omega}_{{(n)}})
\,\,\text{and}\,
\underline{v} =\underline{0} \,\, {\mbox{a.e. on }} {\Omega}^{*}_{(n)}
\setminus {\hat{\Omega}_{(n)}}\}\,\subset\,
{\underline{\bb V}}\hspace*{.1cm}({\Omega}^{*}_{(n)})\,\,,\, \nonumber
\end{align}
respectively
\begin{align} \label{subsetsS}
{\underline{\bb L}}_{{\;\!}{2}}
({{\tilde{\Omega}^{{\!{~}}^{\scriptstyle{*}}}}
}_{\hspace*{.1cm}\!\tilde{\underline{x}},R,(n)},\underline{0})\,:=\,&
\{{\underline{v}}\,\in\,{\underline{\bb L}}_{{\;\!}{2}}({{\tilde{\Omega}^{{\!{~}}^{\scriptstyle{*}}}}
}_{\hspace*{.1cm}\!\tilde{\underline{x}},R,(n)}): \underline{v} =\underline{0} \,\, \,{\mbox{a.e. on }} 
{{\tilde{\Omega}^{{\!{~}}^{\scriptstyle{*}}}}
}_{\hspace*{.1cm}\!\tilde{\underline{x}},R,(n)}
\setminus {{\Omega}_{(n)}}\}\,\subset\,{\underline{\bb L}}_{{\;\!}{2}}({{\tilde{\Omega}^{{\!{~}}^{\scriptstyle{*}}}}
}_{\hspace*{.1cm}\!\tilde{\underline{x}},R,(n)})\,\,\,\,\quad\,\text{and}\,\nonumber\\[-.2cm]
{~} & {~}\\[-.2cm]
{\underline{\bb V}} \hspace{.1cm}({{\tilde{\Omega}^{{\!{~}}^{\scriptstyle{*}}}}}_{\hspace*{.1cm}\!\tilde{\underline{x}},R,(n)},\underline{0})\,:=\,&
\{{\underline{v}}\,\in\,{\underline{\bb V}}\hspace{.1cm}({{\tilde{\Omega}^{{\!{~}}^{\scriptstyle{*}}}}
}_{\hspace*{.1cm}\!\tilde{\underline{x}},R,(n)}): 
{\underline{v}}_{|{\Omega}_{(n)}}\,\in\,{\underline{\bb V}}
\hspace{.1cm}({\Omega}_{{(n)}})
\,\,\text{and}\,
\underline{v} =\underline{0} \,\, {\mbox{a.e. on }} {{\tilde{\Omega}^{{\!{~}}^{\scriptstyle{*}}}}
}_{\hspace*{.1cm}\!\tilde{\underline{x}},R,(n)}
\setminus {{\Omega}_{(n)}}\}\,\subset\,
{\underline{\bb V}}\hspace*{.1cm}({{\tilde{\Omega}^{{\!{~}}^{\scriptstyle{*}}}}
}_{\hspace*{.1cm}\!\tilde{\underline{x}},R,(n)})\,\,.\nonumber
\end{align}
We show the first inequality of \eqref{mon_Stokes__12} as an example of the technique of proving,
regarding that ${\Omega}^{*}_{(n)}\,\not=\,\hat{\Omega}_{{(n)}}$ by $\hat{\Omega}_{{(n)}}\,\subset\,{\Omega}^{*}_{(n)}$:
\begin{align}\label{proof_mon_Stokes_Einskug}
\lambda_{1,S,(n)}({\hat{\Omega}}_{(n)})&\,=\,
\min_{\underline{u} \in {\underline{\bb V}}\hspace{.0cm}
    (\hat{\Omega}_{(n)}):
\| \underline{u}  \|_{{\underline{\bb L}_{2}(.)}}=1 }{\;\!}
    \int_{{\hat{\Omega}}_{(n)~}} \big(
    \sum_{j=1}^{n}\sum_{j=1}^{n}\,\left(\frac{\partial u_j}{\partial x_k}\right)^2\big)
   d{\underline{x}} 
\,=\,\nonumber\\
{~}& \,=\,  
\min_{\underline{u} \in {\underline{\bb V}}\hspace{.0cm}
({\Omega}^{*}_{(n)},\underline{0}):
\| \underline{u}  \|_{{\underline{\bb L}_{2}(.)}}=1 }{\;\!}
    \int_{{\Omega}^{*}
    _{(n)~}} \big(
    \sum_{j=1}^{n}\sum_{j=1}^{n}\,\left(\frac{\partial u_j}{\partial x_k}\right)^2\big)
   d{\underline{x}} 
\,\geq\, \\
{~}& \,\geq \,
\min_{\underline{u} \in {\underline{\bb V}}\hspace{.0cm}
({\Omega}^{*}_{(n)}):
\| \underline{u}  \|_{{\underline{\bb L}_{2}(.)}}=1 }{\;\!}
    \int_{{\Omega}^{*}
    _{(n)~}} \big(
    \sum_{j=1}^{n}\sum_{j=1}^{n}\,\left(\frac{\partial u_j}{\partial x_k}\right)^2\big)
   d{\underline{x}} 
\,\,= \,\,
   \lambda_{1,S,(n)}({o})\,\quad\,.\nonumber
\end{align}
\end{proof}
\begin{corollary}
\label{Poi_Eigv_Omega}
The Poincar\'e constants $c_{p}({{\Omega}}_{(n)})$ and $c_{p,S}({{\Omega}}_{(n)})$ with respect to the spaces ${\bb
  W}_{2}^{1}\hspace{-.62cm}{~}^{{~}^{{~}^{o}}}\hspace{.2cm}({{\Omega}}_{(n)})$ 
and  ${\underline{\bb V}}({{\Omega}}_{(n)})$ 
are bounded in the following way
\begin{align} \label{R_estim_PC}
 c_{p}({{\Omega}}_{(n)})\,\leq \, R\cdot{\kappa^{-1}_{1,L,(n)}(o)}
\,\, \quad\quad\text{resp.}\,\,\,\quad 
c_{p,S}({{\Omega}}_{(n)})\,\leq \, R\cdot{\kappa^{-1}_{1,S,(n)}(o)}\,=\,R\cdot{\kappa^{-1}_{1,L,(n+2)}(o)}
\,.
\end{align}
There we have used the smallest positive roots $\kappa_{1,L,(n)}(o)$ and $\kappa_{1,S,(n)}(o)\,=\,\kappa_{1,L,(n+2)}(o)$
of the corresponding Besselfunctions
$J_{\frac{n}{2}-1}(.)$ and $J_{\frac{n+2}{2}-1}(.)$.
\end{corollary} 
\begin{proof}
We apply Proposition \ref{EW_Laplace_2} with the inequality for
the first eigenvalues of the Laplacians \eqref{mon_Lapl_Diakug}
\begin{align}
\lambda_{1,L,(n)}({{\Omega}}_{(n)})\,\geq \,\lambda_{1,L,(n)}({{\tilde{\Omega}^{{\!{~}}^{\scriptstyle{*}}}}
}_{\hspace*{.1cm}\!\tilde{\underline{x}},R,(n)})\,\,,\nonumber
\end{align}
Proposition \ref{EW_Stokes_12}  
with the inequality for
the first eigenvalues of the Stokes operators 
\eqref{mon_Stokes__12}
\begin{align}
\lambda_{1,S,(n)}({{\Omega}}_{(n)})\,\geq \,\lambda_{1,S,(n)}({{\tilde{\Omega}^{{\!{~}}^{\scriptstyle{*}}}}
}_{\hspace*{.1cm}\!\tilde{\underline{x}},R,(n)})\,,\nonumber
\end{align}
and the equations \eqref{Transf_1st_Eigenv} from Note \ref{first_eigenv_R}
\begin{align} 
\left({\kappa_{1,L,(n)}}(o)\right)^2\,=\,
{\lambda_{1,L,(n)}}(o)\,=\, R^2\cdot {\lambda_{1,L,(n)}({{\tilde{\Omega}^{{\!{~}}^{\scriptstyle{*}}}}
}_{\hspace*{.1cm}\!\tilde{\underline{x}},R,(n)})} \,\,\quad\, &\quad \nonumber\\[-.2cm]
{~} &\quad;\,\,\forall \, n\in \NN\,,\,1\,<\,n\,<\,\infty\,\,.\nonumber\\[-.2cm]
\left({\kappa_{1,S,(n)}}(o)\right)^2\,=\,
{\lambda_{1,S,(n)}}(o)\,=\, R^2\cdot {\lambda_{1,S,(n)}({{\tilde{\Omega}^{{\!{~}}^{\scriptstyle{*}}}}
}_{\hspace*{.1cm}\!\tilde{\underline{x}},R,(n)})} \,\,\quad\, &\quad \nonumber
\end{align}
A simple calculation leads to
\begin{align} \label{R-Absch_EW}
\lambda_{1,L,(n)}({{\Omega}}_{(n)})\,\geq \,\lambda_{1,L,(n)}({{\tilde{\Omega}^{{\!{~}}^{\scriptstyle{*}}}}
}_{\hspace*{.1cm}\!\tilde{\underline{x}},R,(n)})\,=\,&\frac{{\lambda_{1,L,(n)}}(o)}{R^2}
\,\, \quad\quad\text{resp.}\,\,\,\nonumber\\[-.2cm]
{~} & {~} \\[-.2cm]
\lambda_{1,S,(n)}({{\Omega}}_{(n)})\,\geq \,\lambda_{1,S,(n)}({{\tilde{\Omega}^{{\!{~}}^{\scriptstyle{*}}}}
}_{\hspace*{.1cm}\!\tilde{\underline{x}},R,(n)})\,=\,&\frac{{\lambda_{1,S,(n)}}(o)}{R^2}\,=\,
\frac{{\lambda_{1,L,(n+2)}}(o)}{R^2}\,\,.\nonumber
\end{align}
In the last relation of \eqref{R-Absch_EW} the shift Theorem
\ref{thm_shift_Lapl_Stokes_n} comes in handy.
The eigenvalues ${\lambda_{1,L,(n)}}(o)$ and ${\lambda_{1,S,(n)}}(o)={\lambda_{1,L,(n+2)}}$ are the squares of
the smallest positive roots of the corresponding Besselfunctions $J_{\frac{n}{2}-1}(.)$ and $J_{\frac{n+2}{2}-1}(.)$:
$\kappa_{1,L,(n)}(o)$ and
$\kappa_{1,S,(n)}(o)\,=\,\kappa_{1,L,(n+2)}(o)$.  
That means
$J_{\frac{n}{2}-1}(\kappa_{1,L,(n+2)})\,=\,0$ and
$J_{\frac{n}{2}}(\kappa_{1,S,(n)})\,=\,J_{\frac{n+2}{2}-1}(\kappa_{1,L,(n+2)})\,=\,0\,$ 
(cf. also Remark \ref{Besselzeros}).
So for the Poincar\'e constants $c_{p}({{\Omega}}_{(n)})$ and
$c_{p,S}({{\Omega}}_{(n)})$ we find
\begin{align} \label{R_absch__proof_PC}
 c_{p}({{\Omega}}_{(n)})\,\leq \, c_{p}({{\tilde{\Omega}^{{\!{~}}^{\scriptstyle{*}}}}
}_{\hspace*{.1cm}\!\tilde{\underline{x}},R,(n)})\,=\,&\,R \cdot c_{p}({{\Omega^*}}_{(n)})\,=\,\frac{R}{\kappa_{1,L,(n)}(o)}
\,\, \quad\quad\text{resp.}\,\,\,\nonumber\\[-.2cm]
{~} & {~} \\[-.2cm]
c_{p,S}({{\Omega}}_{(n)})\,\leq \,c_{p,S}({{\tilde{\Omega}^{{\!{~}}^{\scriptstyle{*}}}}
}_{\hspace*{.1cm}\!\tilde{\underline{x}},R,(n)})\,=\,&R \cdot c_{p,S}({\Omega}^{*}_{(n)})
\,=\,\frac{R}{\kappa_{1,L,(n+2)}(o)}
\,.\nonumber
\end{align}
\end{proof}
\begin{remark}\label{Poi_Eigv_OmegaII}  
One can prove Corollary \ref{Poi_Eigv_Omega} for the 
Poincar\'e constants using the transformation $F(.)$ (cf. Note \ref{transformationR}) also.
There we take the
radius  $R\,=\, \frac{1}{2}\text{diam}\,{\Omega}_{(n)}\,$ of \,$\,{{\tilde{\Omega}^{{\!{~}}^{\scriptstyle{*}}}}
}_{\hspace*{.1cm}\!\tilde{\underline{x}},R,(n)}$ as the significant 
scaling factor at the transformation $F(.)$, where $F(.)$ maps the closed ball
${\overline{\tilde{\Omega}^{{\!{~}}^{\scriptstyle{*}}}}
}_{\hspace*{-.1cm}\!\tilde{\underline{x}},R,(n)}$ onto the closed unit ball 
${\overline{\White{\hat\Omega}^{{\!{~}}^{\scriptstyle{*}}}\hspace*{-.45cm}
\Omega}}_{(n)}$ (cf. \eqref{transf}). Analogously to \eqref{TransfEigenv}
this yields
\begin{align}\label{TransfEigenvOmega} 
{\lambda_{j,L,(n)}}({\hat\Omega}_{(n)})\,=\, R^2\cdot {\lambda_{j,L,(n)}({\Omega}_{(n)})} \,\,\quad\, &\quad \nonumber\\[-.2cm]
{~} &\quad\forall \, j\in \NN\,.\\[-.2cm]
{\lambda_{j,S,(n)}}({\hat\Omega}_{(n)})\,=\, R^2\cdot  {\lambda_{j,S,(n)}({\Omega}_{(n)})} \,\,\quad\, &\quad \nonumber
\end{align}
Finally we have to use Proposition \ref{EW_Laplace_1} (namely, relation \eqref{mon_Lapl_Einskug}) and the first relation of
\eqref{mon_Stokes__12} in Proposition \ref{EW_Stokes_12}.
\end{remark}
\subsection{Lower and upper Bounds of the Poincar\'e Constants for $\,{\Omega}_{(n)}$ and Evaluation of our Results} \label{sec_eval_results}
It is obvious that the quality of the estimates for the Poincar\'e
constants depends on the relation of the  
volumes $|{\Omega}_{(n)}|$ and $|{\tilde{\Omega}^{{\!{~}}^{\scriptstyle{*}}{\sc  }}
}_{\hspace*{.1cm}\!\tilde{\underline{x}},R,(n)}|$  resp. $|\hat{\Omega}_{{(n)}}|$ and $|{\Omega}^{*}_{(n)}|$. 
We will confirm that the Poincar\'e constants 
$c_{p}({{\Omega}}_{(n)})$ and $c_{p}({{\tilde{\Omega}^{{\!{~}}^{\scriptstyle{*}}}}
}_{\hspace*{.1cm}\!\tilde{\underline{x}},R,(n)})$ are "almost identical" for\, $|{\Omega}_{(n)}|$ $\approxeq$ 
$|{{\tilde{\Omega}^{{\!{~}}^{\scriptstyle{*}}}}
}_{\hspace*{.1cm}\!\tilde{\underline{x}},R,(n)}|$ resp. that the Poincar\'e constants $c_{p}(\hat{\Omega}_{{(n)}})$ and $c_{p}({{\Omega}^{*}}_{(n)})$ are "almost identical" for \, 
$|\hat{\Omega}_{{(n)}}|$ $\approxeq$  $|{\Omega}^{*}_{(n)}|$.
From another point of view we note, that our estimates are closely related and an improvement of the estimate of Poincar\'e constants in
\cite[formula (7.44)]{GilTru} written in our notations as:
\begin{align} \label{GiTru}
c_{p}({{\Omega}}_{(n)})\,\leq \,\left( \frac{|{\Omega}_{(n)}|}{{|\Omega^*_{(n)}|}}
 \right)^{\frac{1}{n}} \,.
\end{align}
\begin{note} \label{RefGiTru} We illustrate our improvement by an example. 
Let ${{\Omega}}_{(n)}:={{\tilde{\Omega}^{{\!{~}}^{\scriptstyle{*}}}}
}_{\hspace*{.1cm}\!\tilde{\underline{x}},R,(n)}$ and let 
the values of $\kappa_{1,L,(n)}(o)$ for ${\Omega}_{(n)}^{*}$ be the
first roots of the corresponding Besselfunctions for $n=2,3,4,\dots$. 
Then \eqref{GiTru} leads to $c_p({{\Omega}}_{(n)})\,\leq \,( \frac{|{\Omega}_{(n)}|}{{|\Omega^*_{(n)}|}}
 )^{\frac{1}{n}}\,=\,R$. 
But taking into account our Corollary
\ref{Poi_Eigv_Omega} for every fixed
$n\,\in\,{\bb N}:\,n>1$ by \eqref{R_estim_PC} we have \quad
$c_{p}({{\Omega}}_{(n)})\,\leq \, R\cdot{\kappa^{-1}_{1,L,(n)}(o)}$\quad
with ${\kappa^{-1}_{1,L,(n)}(o)}\,<\,1$.
\end{note}
\noindent Our method is also suitable to place the Poincar\'e constant
within an interval. In what follows we only consider 
the Poincar\'e constants $c_{p}(.)$ with respect to the space 
${\bb W}_{2}^{1}\hspace{-.62cm}{~}^{{~}^{{~}^{o}}}\hspace{.2cm}(.)$ 
resp. ${\bb
  W}_{2}^{1}\hspace{-.62cm}{~}^{{~}^{{~}^{o}}}\hspace{.2cm}({\breve{\Omega}})$. The
statements for 
the Poincar\'e constants $c_{p,S}(.)$ are adaptions of the results in Corollary
\ref{Poi_Eigv_Omega} to the sets
\begin{align} \label{G_Encl}
\tilde{\underline{x}}\,+\,{\Omega}^{*}_{R,(n)}\,=\,
{{\tilde{\Omega}^{{\!{~}}^{\scriptstyle{*}}}}
}_{\hspace*{.1cm}\!\tilde{\underline{x}},R,(n)}\,\supset \,
{{\Omega}}_{(n)}\,\supset \,\check{\underline{x}}\,+\,{\Omega}^{*}_{R_i,(n)}\,=\,
{{\check{\Omega}^{{\!{~}}^{\scriptstyle{*}}}}
}_{\hspace*{.1cm}\!\check{\underline{x}},R_i,(n)}\,\,\,,
\end{align}
where we denote by ${{\check{\Omega}^{{\!{~}}^{\scriptstyle{*}}}}
}_{\hspace*{.1cm}\!\check{\underline{x}},R_i,(n)}$ an open ball ${{\check{\Omega}^{{\!{~}}^{\scriptstyle{*}}}}
}_{\hspace*{.1cm}\!\check{\underline{x}},R_i,(n)}
:=\,
\check{\underline{x}}\,+\,{\Omega}^{*}_{R_i,(n)}\,:=\{{\underline{x}} \in {\bb R}^{n}:\,
\|{\underline{x}}-\check{\underline{x}}\|\,<\,R_i\,\}$ with the maximal radius $R_i$ such that
${{\Omega}}_{(n)}\,\supset \,\check{\underline{x}}\,+\,{\Omega}^{*}_{R_i,(n)}\,=\,
{{\check{\Omega}^{{\!{~}}^{\scriptstyle{*}}}}
}_{\hspace*{.1cm}\!\check{\underline{x}},R_i,(n)}$. 
We note, that the choice for 
${{\check{\Omega}^{{\!{~}}^{\scriptstyle{*}}}}
}_{\hspace*{.1cm}\!\hat{\underline{x}},R_i,(n)}$ does not need to be unique. The
existence of such a ball is ensured by the 
convexity of ${{\Omega}}_{(n)}$. For any fixed arbitrary ball with
maximal radius $R_i$ we see
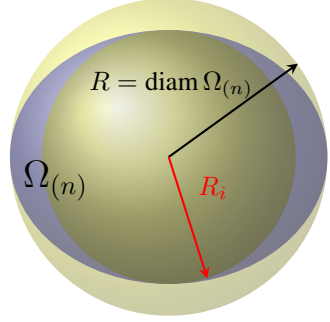
\begin{figure}[ht]
{~}\hspace*{6.05cm}
\begin{tikzpicture}[scale=1.2,tdplot_main_coords] 
\coordinate (P) at ({0},{2.2225/sqrt(2)},{2.2225/sqrt(2)
});
\coordinate (P1) at ({0},{0.47},{-1.43}
);
\shade[ball color = yellow,
    opacity = 0.35
] (0,0,0) circle (1.75cm);
\shade[ball color = blue,
    opacity = 0.4
] (0,0,0) ellipse (1.75cm and 1.4cm);
\shade[ball color = yellow,
    opacity = 0.4
] (0,0,0) circle (1.4cm);
\draw[thick,black, -stealth] (0,0,0) -- (P);
\draw[thick,red, -stealth] (0,0,0) -- (P1);
\draw(-.3,1.,1.)  node[anchor=east,black]{$R=\text{diam}\,\Omega_{(n)}$};
\draw(.15,0.9,-.1)  node[anchor=east,black]{\red {$R_i$}};
\draw(0.25,-0.75,-0.35) node[anchor=east]{\Large{$\Omega_{(n)}$}};
%
\end{tikzpicture}\\
\caption{$2$d Sketch of ${{\Omega}}_{(n)}$ included in  
$\tilde{\underline{x}}\,+\,{\Omega}^{*}_{R,(n)}$ (ball with radius  $R=\text{diam}\,\Omega_{(n)}$) and 
$\check{\underline{x}}\,+\,{\Omega}^{*}_{R_i,(n)}$  included in ${{\Omega}}_{(n)}$}
\label{fig:balldomballgeometry}
\end{figure}
\begin{corollary}
\label{Poi_Encl_Omega} Let ${{\Omega}}_{(n)}$ fulfill relation
\eqref{G_Encl}, i.e. 
${{\tilde{\Omega}^{{\!{~}}^{\scriptstyle{*}}}}
}_{\hspace*{.1cm}\!\tilde{\underline{x}},R,(n)}\,\supset \,
{{\Omega}}_{(n)}\,\supset \,
{{\check{\Omega}^{{\!{~}}^{\scriptstyle{*}}}}}_{\hspace*{.1cm}\!\check{\underline{x}},R_i,(n)}$.
Then the Poincar\'e constants $c_{p}({{\Omega}}_{(n)})$ with respect to 
${\bb W}_{2}^{1}\hspace{-.62cm}{~}^{{~}^{{~}^{o}}}\hspace{.2cm}({{\Omega}}_{(n)})$ 
(cf. \eqref{Poinc}) are bounded from above and below by
\begin{align} \label{R_Ri_estim_PC}
R_i\cdot{\kappa^{-1}_{1,L,(n)}(o)}\,\leq \, 
 c_{p}({{\Omega}}_{(n)})\,\leq \, R\cdot{\kappa^{-1}_{1,L,(n)}(o)}.
\end{align}
There we have used the smallest positive roots $\kappa_{1,L,(n)}(o)$ 
of the corresponding Besselfunctions
$J_{\frac{n}{2}-1}(.)$. Explicitly,
\begin{align}\label{mon_Lapl_Kug_Diakug}
{R_i }^{-2}\cdot{\kappa^{2}_{1,L,(n)}(o)}
\,=\,
\lambda_{1,L,(n)}(
{{\check{\Omega}^{{\!{~}}^{\scriptstyle{*}}}}
}_{\hspace*{.1cm}\!\check{\underline{x}},R_i,(n)})\,\geq \,
\lambda_{1,L,(n)}({{\Omega}}_{(n)})\,\geq \,\lambda_{1,L,(n)}({{\tilde{\Omega}^{{\!{~}}^{\scriptstyle{*}}}}
}_{\hspace*{.1cm}\!\tilde{\underline{x}},R,(n)})\,=\,R^{-2}\cdot{\kappa^{2}_{1,L,(n)}(o)}.
\end{align}
\end{corollary}
\begin{proof}
We use the same ideas as in the proofs of 
Propositions \ref{EW_Laplace_1} and \ref{EW_Laplace_2} and exchange the roles in the inclusions
for the additionally inequalities.
\end{proof}
\begin{figure}[ht]
\hspace*{6.92cm}
\begin{tikzpicture}[x  = {(0.5cm,0.5cm)},
                 y  = {(0.95cm,-0.25cm)},
                  z  = {(0cm,0.9cm)}]
\begin{scope}[canvas is yz plane at x=-1]
  \shade[left color=dark-blue!50,right color=blue!20] (-1,-1) rectangle (1,1);
\end{scope}
\begin{scope}[canvas is xz plane at y=1]
  \shade[right color=dark-blue!70,left color=blue!20] (-1,-1) rectangle (1,1);
\end{scope}
\begin{scope}[canvas is yx plane at z=1]
  \shade[top color=dark-blue!80,bottom color=blue!20] (-1,-1) rectangle (1,1);
\end{scope}
\end{tikzpicture}\\[-3.5cm]
\hspace*{6.54cm}
\begin{tikzpicture}[scale=1.852,tdplot_main_coords] 
\coordinate (P) at ({0},{1.905/sqrt(2)},{1.905/sqrt(2)
});
\shade[ball color = yellow,
    opacity = 0.25
] (0,-0.5,0) circle (1.005cm);
\shade[ball color = Orchid,
    opacity = 0.4
] (0,-0.5,0) circle (0.7571cm);
\end{tikzpicture}\\[-.5cm]
\caption{$2$d Sketch of a cube  ${\hat{\Omega}}_{(n)}$ as a $C^{0,1}$ domain example with
${\hat{\Omega}}_{(n)}\,\subset {\Omega}_{(n)}^{*} $ and the ball 
${\Omega}^{*}_{R_i,(n)}$  included in ${\hat{\Omega}}_{(n)}$ }
\label{fig:cubegeometry}
\end{figure}
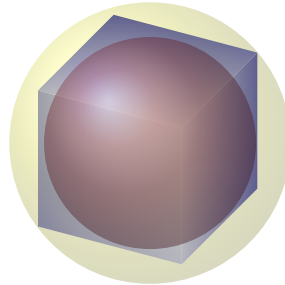
{~}\\[.05cm]
\noindent The last step in our investigations is the evaluation of our results on a cube in ${\bb R}^{n}$. We choose
${{\Omega}}_{(n)}\,:=\,
{\hat{\Omega}}_{(n)}\,\subset\,{\Omega}^{*}_{(n)}$ for the sake of simplicity.
There we note that a cube ${{\Omega}}_{(n)}\,=\,{\hat{\Omega}}_{(n)}$
is a $C^{0,1}$-domain with known first eigenvalues. For every fixed
$n\,\in\,{\bb N}:\,n>1$, the edge length $a_{(n)}$ of a cube with
$R\,=\,1\,=\,\text{diam}\,{{\Omega}}_{(n)}\,=\,\text{diam}\,{\hat{\Omega}}_{(n)}$
is given by Pythagoras's theorem as $a_{(n)}\,=\,\sqrt{\frac{4}{n}}$.  
The volume of ${\Omega}^{*}_{(n)}$ was given in the {General notation A} as 
$|{\Omega}^{*}_{(n)}|\,=\,
\frac{{2\pi}^{\frac{n}{2}}}{\Gamma(\frac{n}{2})\cdot n} \,\,\,\text{and the inside ball\,}\,{\Omega}^{*}_{R_i,(n)}\,=\,
{{\check{\Omega}^{{\!{~}}^{\scriptstyle{*}}}}
}_{\hspace*{.1cm}\!{\underline{0}},R_i,(n)} \,\,\text{has the volume\,}|{\Omega}^{*}_{R_i,(n)}|\,=\,(\frac{a_{(n)}}{2})^n\cdot|{\Omega}^{*}_{(n)}|.$
The radius $R_i\,=\,\frac{a_{(n)}}{2}$. 
The volume of the cube ${\hat{\Omega}}_{(n)}$ is 
$|{\hat{\Omega}}_{(n)}|\,=\,a_{(n)}^n$ and the first eigenvalues are 
$\lambda_{1,L,(n)}({\hat{\Omega}}_{(n)})\,=\,n\cdot (\frac{\pi}{a_{(n)}})^2$.\\[2mm]
The volumes $|{\Omega}^{*}_{(n)}|$, $|{\hat{\Omega}}_{(n)}|$ and 
$|{\Omega}^{*}_{R_i,(n)}|$, respectively, and the Poincar\'e constants $c_{p}(.)$ for 
${\Omega}^{*}_{(n)}$, ${\hat{\Omega}}_{(n)}$ and 
${\Omega}^{*}_{R_i,(n)}$ (depending on $n$) are easy to calculate as\\[.2cm]
\begin{tabular}{|c|r|r|r|r|r|r|}
\hline 
$n$  & $|{\Omega}^{*}_{(n)}|$ &  $|{\hat{\Omega}}_{(n)}|$ & $|{\Omega}^{*}_{R_i,(n)}|$ &
$c_{p}({\Omega}^{*}_{(n)})$
& $c_{p}({\hat{\Omega}}_{(n)})$ & $c_{p}({\Omega}^{*}_{R_i,(n)})$ \\
\hline
2 & 3.141592654 &  2.000000000  &  1.570796326 & 0.4158305773 & 0.3183098861 &  0.2940366210
\\
\hline
3 & 4.188790204 &  1.539600716  &  0.8061330502 & 0.3183098861 & 0.2122065907 & 0.1837762984 
\\
\hline
4 & 4.934802202 &  1.000000000  &  0.3084251376 & 0.2609803591 & 0.1591549431 & 0.1304901796
\\
\hline
5 & 5.263789015 & 0.5724334022  &  0.09416152044 & 0.2225481584 & 0.1273239544 & 0.09952656210
\\
\hline
6 & 5.167712783 &  0.2962962962  &  0.02392459620 & 0.1947183693 & 0.1061032954 & 0.07949344140
\\
\hline
7 & 4.724765972 &  0.1410479666  &  0.005206395568 & 0.1735069107 & 0.09094568173 & 0.06557944805
\\
\hline
8 & 4.058712129 &  0.06250000001  & 0.0009908965158 & 0.1567358347 & 0.07957747155 & 0.05541448575
\\
\hline
9 & 3.298508904 &  0.02601229485 & 0.0001675816136 & 0.1431038539 & 0.07073553025 & 0.04770128463
\\
\hline
\end{tabular}
{~}\\[.25cm]
The statements for the (Stokes-)Poincar\'e constants $c_{p,S}(.)$ for a cube in ${\bb R}^{n}$ are very difficult.
The simple reason is that a cube (or a square) ${{\Omega}}_{(n)}\,=\,{\hat{\Omega}}_{(n)}$ in 
${\bb R}^{n}$ is a $C^{0,1}$-domain. The condition for the vanishing
divergence for the eigenfunctions impede a simple separation of the
variables there. 
We consider the square ${{\Omega}}_{(2)}\,=\,{\hat{\Omega}}_{(2)}$ in 
${\bb R}^{2}$ as an example. With the help of appropriate numerical
methods the approximation for the first Stokes  
eigenvalue ${\lambda_{1,S,(2)}}({\hat{\Omega}}_{(2)}) \,\approxeq \,10.600725212$
(cf. \cite{RumSqare}). So we get approximately  $c_{p,S}({\hat{\Omega}}_{(2)})\,\approx\,0.3183098861$.
But these results are restricted to small $n\,\in\,{\bb N}$.

\section*{Appendix}\label{App_Bes}
Let the Cartesian unit vectors of ${\bb
  R}^{n}\;(n\geq 3)$
be given by ${\underline{\mathfrak{e}}}_{j}\,:=\,
(\delta_{j,1},\delta_{j,2}, \dots, \delta_{j,n})^{T}$ for all $j=1,2,\dots,n$, with
Kronecker's delta $\delta_{j,k}$. 
The polar coordinates are stated as $r$,  $\vartheta_{1}$, $\dots$ , $\vartheta_{n-2}$ and $\varphi$ 
(cf. the following Definition \ref{polarcoord}) with the corresponding unit vectors
${\underline{\mathfrak{e}}}_{{\;\!}r}$, ${\underline{\mathfrak{e}}}_{{\;\!}\vartheta_{1}}$, $\dots$, ${\underline{\mathfrak{e}}}_{{\;\!}\vartheta_{n-2}}$  and
${\underline{\mathfrak{e}}}_{{\;\!}\varphi}$.
\begin{definition} \label{polarcoord}
The representation of any point ${\underline{x}} \in {\bb R}^{n}$ in the system
of polar coordinates is given via
\begin{align}  \label{Pol1Sh}
\begin{array}{rcl}
x_1 &\,=\,&r\cdot \sin{\vartheta_{1}}\cdots\sin{\vartheta_{n-2}}\cos{\varphi}\\
x_2 &\,=\,&r\cdot \sin{\vartheta_{1}}\cdots\sin{\vartheta_{n-2}}\sin{\varphi}\\
x_3 &\,=\,&r\cdot \sin{\vartheta_{1}}\cdots\cos{\vartheta_{n-2}}\\
{~} & \vdots  & {~} \\
x_{n-1} &\,=\,&r\cdot \sin{\vartheta_{1}}\cos{\vartheta_{2}}\\
x_{n} &\,=\,&r\cdot \cos{\vartheta_{1}}
\end{array}\\
\mbox{with}\,\,
r\,:=\,\|{\underline{x}}\| \in [0,\infty) ,\,\vartheta_{1},\dots,\vartheta_{n-2}\,\in\,[0,\pi]
\,\,\mbox{and}\,\,\varphi\,\in\,[0,2\pi] \,\,.\nonumber
\end{align}
\end{definition}
\begin{notation}[Surface harmonics of degrees $\ell = 0$ and $\ell = 1$]\label{SurfHarm0+1}
The function $f(\vartheta_{n-2},\dots\,\vartheta_{1},\varphi)\,=\,const.\,\not=\,0$  is a non-vanishing harmonic 
polynomial of degree $\ell = 0$ in $r$. For any $ n\,\in {\bb
  N}:n\,>\,1$ the sherical  
surface harmonic function
\begin{align}  \label{SphSurf0}
S^{\{0\}}\,:=\,1\,.
\end{align}
In the representation \eqref{Pol1Sh} the functions $\{x_{k}\}_{k=1}^n$
are harmonic polynomials
of degree $\ell = 1$ in $r$. We call the functions
 \begin{align}  \label{SphSurf1}
\begin{array}{rcl}
S^{\{1\}}_1 &\,=\,& \sin{\vartheta_{1}}\cdots\sin{\vartheta_{n-2}}\cos{\varphi}\\
S^{\{1\}}_2 &\,=\,& \sin{\vartheta_{1}}\cdots\sin{\vartheta_{n-2}}\sin{\varphi}\\
S^{\{1\}}_3 &\,=\,& \sin{\vartheta_{1}}\cdots\cos{\vartheta_{n-2}}\\
{~} & \vdots  & {~} \\
S^{\{1\}}_{n-1} &\,=\,& \sin{\vartheta_{1}}\cos{\vartheta_{2}}\\
S^{\{1\}}_{n} &\,=\,& \cos{\vartheta_{1}}\,,
\end{array}\\
\mbox{where}\,\,
\,\vartheta_{1},\dots,\vartheta_{n-2}\,\in\,[0,\pi]
\,\,\mbox{and}\,\,\varphi\,\in\,[0,2\pi] \,\,,\nonumber
\end{align}
sherical surface harmonics of degree $\ell = 1$. We write $S\,\in\,{\mbox{span}}\{S^{\{1\}}_k\}_{k=1}^n$ 
for a spherical surface harmonic function of degree $\ell = 1$
(e.g. the Definition in Subsection  6.3.1 in \cite{Triebel}) as well.
\end{notation}
\begin{remark}\label{JacobiMat}
The first step in the calculation of
the transformation between Cartesian and spherical polar coordinates
(cf. Definition \ref{polarcoord})
is to calculate the Jacobian matrix:
\begin{align*}
{\underline{\underline{J}}} =
\begin{bmatrix} 
\frac{\partial x_1}{\partial r}&\frac{\partial x_1}{\partial \vartheta_{1}}&
\dots &\frac{\partial x_1}{\partial \vartheta_{n-2}}
&\frac{\partial x_1}{\partial{\varphi}}\\
\vdots&\vdots & {~}& \vdots&\vdots\\
\frac{\partial x_n}{\partial r}&\frac{\partial x_n}{\partial \vartheta_{1}}&
\dots  &\frac{\partial x_n}{\partial \vartheta_{n-2}}&\frac{\partial x_n}{\partial{\varphi}}
\end{bmatrix}
  \end{align*}
 and in a second step the corresponding metric tensor 
 ${\underline{\underline{g}}}$:
\begin{align*}  
{\underline{\underline{g}}}\,:=\,{\underline{\underline{J}}} \cdot {\underline{\underline{J}}}^T\,=\,
\text{diag}\{g_{j,j}\}_{j=1}^{n}\quad\,\text{with}\,\, g_{1,1}=1,\,g_{2,2}=r^2,\,\dots\,,
g_{n,n}=r^2{\sin}^2{\vartheta_{1}}  \cdots {\sin}^2{\vartheta_{n-2}}
\,.
\end{align*}
The reciprocal of the square roots of the $\{g_{j,j}\}_{j=1}^{n}$ used
as multipliers column by column 
applied on ${\underline{\underline{J}}}$ provide  the transformation matrices.
\end{remark}
\begin{remark}
\label{R1} The transformation between Cartesian and spherical polar coordinates
(cf. Notation \ref{N3}) as the transformation of one coordinate system to the other 
can be written as ${\underline{u}}_{\mathfrak{c}}\,=\,
{\underline{\underline{T}}}_{{\mathfrak{c}},{\mathfrak{s}}}{\underline{u}}_{\mathfrak{s}}$
\,or\,\,\,${\underline{u}}_{\mathfrak{s}}\,=\,
{\underline{\underline{T}}}^{-1}_{{\mathfrak{c}},{\mathfrak{s}}}
{\underline{u}}_{\mathfrak{c}}\,=\,
{\underline{\underline{T}}}_{{\mathfrak{s}},{\mathfrak{c}}}{\underline{u}}_{\mathfrak{c}}$,
respectively,
where we  use the concept of columns of coordinates
and the transformation matrices ${\underline{\underline{T}}}_{{\mathfrak{s}},{\mathfrak{c}}}\,:=\,
{\underline{\underline{T}}}_{{\mathfrak{c}},{\mathfrak{s}}}^{-1}\,=\,
{\underline{\underline{T}}}_{{\mathfrak{c}},{\mathfrak{s}}}^{T}\,$.\\
For general $n$:
\begin{equation}\label{R1_T=n} 
		{\underline{\underline{T}}}_{{\mathfrak{c}},{\mathfrak{s}}}\,
		:=\,\left[
\begin{array}{llccc}
\sin{\vartheta_{1}}\sin{\vartheta_{2}}\cdots \sin{\vartheta_{n-2}}\cos{\varphi} & 
\cos{\vartheta_{1}}\sin{\vartheta_{2}}\cdots \sin{\vartheta_{n-2}}\cos{\varphi}& \dots &
\cos{\vartheta_{n-2}}\cos{\varphi}&
-\sin{\varphi} \\
\sin{\vartheta_{1}}\sin{\vartheta_{2}}\cdots \sin{\vartheta_{n-2}}\sin{\varphi} & 
\cos{\vartheta_{1}}\sin{\vartheta_{2}}\cdots \sin{\vartheta_{n-2}}\sin{\varphi} & \dots &\cos{\vartheta_{n-2}}\sin{\varphi}&
 {~} \cos{\varphi}\\
\sin{\vartheta_{1}}\sin{\vartheta_{2}}\cdots \cos{\vartheta_{n-2}}& 
\cos{\vartheta_{1}}\sin{\vartheta_{2}}\cdots \cos{\vartheta_{n-2}} & \dots & - \sin{\vartheta_{n-2}}&
 {~} 0\\
\vdots & \vdots & \ddots &\vdots & \vdots\\
\sin{\vartheta_{1}}\cos{\vartheta_{2}} & \cos{\vartheta_{1}}\cos{\vartheta_{2}} &\dots &
 {~} 0 & {~} 0\\
\cos{\vartheta_{1}} & -\sin{\vartheta_{1}}& \dots &    {~} 0 & {~} 0
		\end{array}
		\right] 
                \,.
\end{equation}
\end{remark}
\begin{remark}\label{Lapldimn}
The Laplacian in spherical coordinates
$\Delta_{sph}(.):=\Delta_{r,\vartheta_{1},\dots\,\vartheta_{n-2},\varphi}(.)$ is
($n\geq3$; see, e.g.,  6.3.4 in \cite{Triebel})
\begin{align*} \hspace*{-.3cm}
 \Delta_{sph}(.)  = \frac{1}{r^{n-1}}
\frac{\partial{ }}{\partial r}({r^{n-1}}\frac{\partial { (.)}}{\partial r})
+\frac{1}{r^{2}}\left(
\frac{1}{\sin^{n-2}{\vartheta_{1}}}\frac{\partial{ }}{\partial \vartheta_{1}}
(\sin^{n-2}\vartheta_{1}\frac{\partial{ }(.)}{\partial \vartheta_{1}})+
\frac{1}{\sin^{2}\vartheta_{1}\sin^{n-3}{\vartheta_{2}}}\frac{\partial{ }}{\partial \vartheta_{2}}
(\sin^{n-3}\vartheta_{2}\frac{\partial{ }(.)}{\partial \vartheta_{2}})\,+\right.
\\
{~} 
\left.\,\dots\,+\,
\frac{1}{\sin^{2}\vartheta_{1}\sin^{2}{\vartheta_{2}}\cdots
\sin^{2}{\vartheta_{n-3}}
\sin^{~}{\vartheta_{n-2}}
}\frac{\partial{ }}{\partial \vartheta_{n-2}}
(\sin \vartheta_{n-2}\frac{\partial{ }(.)}{\partial \vartheta_{n-2}})\,+
\,
\frac{1}{\sin^{2}\vartheta_{1}\sin^{2}{\vartheta_{2}}\cdots
\sin^{2}{\vartheta_{n-2}}}
  \frac{\partial^{2}{ (.)}}{\partial \varphi^{2}}\right)\,\,,
\end{align*}
resp. $\displaystyle
\Delta_{r,\vartheta_{1},\dots\,\vartheta_{n-2},\varphi}(.)  = 
{~} 
\,\frac{1}{r^{n-1}}
\frac{\partial{ }}{\partial r}({r^{n-1}}\frac{\partial { (.)}}{\partial r})
\,-\,\frac{\displaystyle{1}}{\displaystyle{r^{2}}}{ \mbox{B}}{\;\!}(.)\,,\hspace*{.2cm}{~}
$
where B$(.) $ denotes the Beltrami differential operator.
\end{remark}
\noindent We define the Laplace-Beltrami operator by means of
Beltrami's differential operator
in the following
\begin{definition}\label{DBeltr} { For all $Y\,\in\,D({\boldsymbol
  B^{\circledast}})=C^{\infty}({\omega_{(n)}}) \subset {\bb
  L}_{2}({\omega_{(n)}})\,$  the Laplace-Beltrami operator is defined
as}
\begin{align*}
{\boldsymbol B^{\circledast}}\,{Y} := \,{ \mbox{B}}{\;\!}(Y)\,\,=\,-
\left(
\frac{1}{\sin^{n-2}{\vartheta_{1}}}\frac{\partial{ }}{\partial \vartheta_{1}}
(\sin^{n-2}\vartheta_{1}\frac{\partial{ } Y}{\partial \vartheta_{1}})+
\frac{1}{\sin^{2}\vartheta_{1}\sin^{n-3}{\vartheta_{2}}}\frac{\partial{ }}{\partial \vartheta_{2}}
(\sin^{n-3}\vartheta_{2}\frac{\partial{ } Y }{\partial \vartheta_{2}})\,+\right.\hspace*{3.1cm}  
\\
{~} 
\left.\,\dots\,+\,
\frac{1}{\sin^{2}\vartheta_{1}\sin^{2}{\vartheta_{2}}\cdots
\sin^{2}{\vartheta_{n-3}}
\sin^{~}{\vartheta_{n-2}}
}\frac{\partial{ }}{\partial \vartheta_{n-2}}
(\sin \vartheta_{n-2}\frac{\partial{ } Y}{\partial \vartheta_{n-2}})\,+
\,
\frac{1}{\sin^{2}\vartheta_{1}\sin^{2}{\vartheta_{2}}\cdots
\sin^{2}{\vartheta_{n-2}}}
\frac{\partial^{2}{\, Y }}{\partial \varphi^{2}}\right)\,\,.
\end{align*}
We denote the Friedrichs' extension of ${\boldsymbol B^{\circledast}}$
by ${\boldsymbol 
  B}:={\overline{\boldsymbol B^{\circledast}}}$, where ${\boldsymbol
  B}$ is applied on 
$D({\boldsymbol B})\,:=\,{\bb W}_{2}^{2}({\omega_{(n)}})\subset {\bb L}_{2}({\omega_{(n)}})$.
\end{definition}
\begin{remark} The detailed construction of the Laplace-Beltrami operator ${\boldsymbol
  B}$ is given in \cite[Subsection 6.3.5]{Triebel} at great length. Especially, the step from an $n$-dimensional shell to the boundary ${\omega_{(n)}}$ is illustrated there.
\end{remark}
\noindent We explicitly quote the following result which is important for the eigenfunction of the Laplacian 
as well as for the Stokes operator.
\begin{theorem} \label{thmsbeltr}
The Laplace-Beltrami operator ${\boldsymbol B}$ is nonnegative and self-adjoint.
${\boldsymbol B}$ is an operator with pure point spectrum. Its eigenvalues are $\ell (\ell+n-2)$,  
$\ell\,=\,0,1,2,\dots $.
The surface harmonics S(.) of the degree $\ell$ form a set of all eigenfunctions of ${\boldsymbol B}$
to the eigenvalue $\ell (\ell+n-2)$.
\end{theorem}
\end{document}